\documentclass[11pt,reqno]{amsart}
\pdfoutput=1
\usepackage{amsthm, amsmath,amsfonts,amssymb,euscript,hyperref,graphics,color,slashed}
\usepackage{graphicx}
\usepackage{mathrsfs}
\usepackage{comment}
\usepackage{import}
\usepackage{tikz}
\usepackage{latexsym}
\usepackage[makeroom]{cancel}
\usepackage[font=normalsize]{caption}
\usepackage{subfigure}
\usepackage[title]{appendix}
\usepackage{enumerate}
\usepackage{mathrsfs}
\usepackage{empheq}   
\usepackage{bm}
\usepackage{bookmark}
\usepackage{arydshln}
\usepackage{lineno,hyperref}
\usepackage{cite}
\usepackage{float}
\usepackage[utf8]{inputenc}
\usepackage{dashbox}
\usepackage{ulem}

\def\inte#1{
\displaystyle\mathop{#1\kern0pt}^\circ }

\def\virgp{\raise 2pt\hbox{,}}
\def\cdotpv{\raise 2pt\hbox{;}}

\def\C{\mathop{\mathbb C\kern 0pt}\nolimits}
\def\DD{\mathop{\mathbb D\kern 0pt}\nolimits}
\def\EE{\mathop{{\mathbb E \kern 0pt}}\nolimits}
\def\K{\mathop{\mathbb K\kern 0pt}\nolimits}
\def\N{\mathop{\mathbb N\kern 0pt}\nolimits}
\def\Q{\mathop{\mathbb Q\kern 0pt}\nolimits}
\def\R{\mathop{\mathbb R\kern 0pt}\nolimits}
\def\SS{\mathop{\mathbb S\kern 0pt}\nolimits}
\def\ZZ{\mathop{\mathbb Z\kern 0pt}\nolimits}
\def\TT{\mathop{\mathbb T\kern 0pt}\nolimits}
\def\P{\mathop{\mathbb P\kern 0pt}\nolimits}

\newcommand{\beq}{\begin{equation}}
\newcommand{\eeq}{\end{equation}}
\newcommand{\ben}{\begin{eqnarray}}
\newcommand{\een}{\end{eqnarray}}
\newcommand{\beno}{\begin{eqnarray*}}
\newcommand{\eeno}{\end{eqnarray*}}

\newtheorem{thm}{Theorem}[section]
\newtheorem{lem}{Lemma}[section]
\newtheorem{rmk}{Remark}[section]

\newtheorem{prop}{Proposition}[section]

\newtheorem*{Main Theorem}{Main Theorem}
\newtheorem{theorem}{Theorem}[section]

\newtheorem{definition}[theorem]{Definition}
\newtheorem{remark}[theorem]{Remark}

\numberwithin{equation}{section}

\allowdisplaybreaks

\title[Ill-posedness for 2D Euler and ideal MHD]{The Cauchy problems for the 2D compressible Euler equations and ideal MHD system are ill-posed in $H^\frac{7}{4}(\mathbb{R}^2)$}
\author{Xinliang An$^*$$^1$}\author{Haoyang Chen$^*$$^2$}\author{Silu Yin$^{\dag}$$^3$}

\date{}

\AtEndDocument{%
  \par
  \medskip
  \begin{tabular}{@{}l@{}}%
    \textsc{$^*$Department of Mathematics, National University of Singapore, Singapore}\\
    \textit{E-mail address}: \texttt{$^1$matax@nus.edu.sg}\\
    \textit{E-mail address}: \texttt{$^2$hychen@nus.edu.sg}\\
    \textsc{$^\dag$School of Mathematics, Hangzhou Normal University, China}\\
    \textit{E-mail address}: \texttt{$^3$yins@hznu.edu.cn}\\
  \end{tabular}}

\begin{document}
\maketitle

\newcommand\blfootnote[1]{%
\begingroup
\renewcommand\thefootnote{}\footnote{#1}%
\addtocounter{footnote}{-1}%
\endgroup
}
\begin{abstract}
In a fractional Sobolev space $H^s(\mathbb{R}^2)$ with $s\leq\frac74$, we prove the low-regularity ill-posedness for the 2D compressible Euler equations and the 2D ideal compressible MHD system. Our ill-posedness results match the $H^\frac74$ regularity threshold for the 2D compressible Euler system with respect to the fluid velocity and density.
\\

\end{abstract}

\tableofcontents
\newpage
\section{Introduction and Main results}
The MHD system describes the dynamics of electrically conducting fluids or plasmas, accounting for the interplay between fluid motion and magnetic fields. Mathematically, it generalizes the compressible Euler equations, recovered when the magnetic field vanishes, and provides a rich example of a nonlinear hyperbolic system with multiple wave speeds and nonlinear coupling. Physically, it models phenomena ranging from astrophysical and solar plasmas to laboratory experiments with conducting fluids, capturing complex interactions between fluid motion and magnetic forces. In this paper, we prove the $H^{\frac{7}{4}-}$ ill-posedness and anisotropic $H^\frac74$ ill-posedness for (local well-posedness of) Cauchy problems of the 2D compressible Euler equations and the 2D ideal compressible MHD system. Both of them are physical systems with multiple wave speeds. The 2D ideal compressible MHD system takes the form:
\begin{equation}\label{MHD}
\left\{\begin{split}
&\partial_t\varrho+\nabla\cdot(\varrho u)=0,\\
&\varrho\{\partial_t+(u\cdot\nabla)\}u+\nabla p+\frac{\mu_0}{2} \nabla |H|^2-\mu_0 H \cdot \nabla H=0,\\
&\partial_t H+(u\cdot \nabla)H-H\cdot \nabla u +H \nabla\cdot u=0,\\
&\partial_t S+(u\cdot\nabla)S=0,
\end{split}
\right.
\end{equation}
where $\mu_0$ is the magnetic permeability constant, $\varrho$ is the fluid density\footnote{In our analysis, the value of $\varrho$ is close to $1$.}, $u=(u_1,u_2) \in \mathbb{R}^2$ is the fluid velocity, $H=(H_1,H_2) \in \mathbb{R}^2$ is the magnetic field intensity, $S$ is the entropy and $p$ is the pressure satisfying the equation of state $p=p(\varrho,S).$ The magnetic field $H$ verifies the following constraint equation:
\begin{equation} \label{constraint H}
\nabla\cdot H=0.
\end{equation}
Taking the divergence of the third equation in \eqref{MHD}, one can verify that once imposed at the initial data level, this constraint is automatically propagated in time. We consider the polytropic gas. The pressure $p(\varrho,S)$ obeys the following equation of state
\begin{equation*}
p=A e^S\varrho^\gamma,
\end{equation*}
where $A$ is a positive constant and $\gamma>1$ is the adiabatic gas constant.

If the magnetic field vanishes, i.e., $H_1=H_2\equiv0$, the above system \eqref{MHD} reduces to the following 2D compressible Euler equations allowing non-trivial entropy and vorticity
\begin{equation}\label{euler1}
\left\{\begin{split}
&\partial_t\varrho+\nabla\cdot(\varrho u)=0,\\
&\varrho\{\partial_t+(u\cdot\nabla)\}u+\nabla p=0,\\
&\partial_t S+(u\cdot\nabla)S=0.
\end{split}
\right.
\end{equation}
For the 2D compressible Euler equations, the low-regularity local well-posedness result was proven by Zhang in \cite{zhang}. In particular, Zhang proved the local well-posedness in $H^s$ for $s>\frac74$ with respect to the regularity of fluid velocity, vorticity and density. Therefore, according to the corresponding low-regularity local well-posedness result by Zhang \cite{zhang}, the Cauchy problems of the 2D compressible Euler equations is anticipated to be ill-posed at the threshold regularity $H^\frac{7}{4}$. See Remark \ref{rk euler} for a more detailed discussion.

In this paper, we first show that the Cauchy problems of the 2D ideal compressible MHD equations \eqref{MHD} are (isotropically) ill-posed in the fractional Sobolev space $H^s(\mathbb{R}^2)$ for $s<\frac74$ and anisotropically ill-posed in $H^\frac74(\mathbb{R}^2)$. Furthermore, this ill-posedness is driven by the instantaneous shock formation. To our best knowledge, this is the first nonlinear result of low regularity ill-posedness for the 2D ideal compressible MHD. In the absence of the magnetic field, the ideal MHD system reduces to the compressible Euler equations. Our work thus also serves as the first low-regularity ill-posedness result for the 2D compressible Euler equations \eqref{euler1}. 

\subsection{Background and history}
Our research is motivated by a series of classic works on low-regularity ill-posedness and shock formation. We first review studies on low-regularity ill-posedness. Under planar symmetry, Lindblad \cite{lindblad93,Lind96,Lind98} constructed sharp counterexamples to the local well-posedness for semilinear and quasilinear wave equations in three dimensions. Low-regularity local well-posedness for the quasilinear wave equation was proven by Smith-Tataru \cite{tataru}. They showed that for $n$ dimensional quasilinear wave equations, the Cauchy problems are locally well-posed in $H^s(\mathbb{R}^n)$ with $s>n/2+3/4$ for $n=2$ and $s>(n+1)/2$ for $n=3,4,5$. For $n=3$, via the vector field method, Wang \cite{Wang} gave a different proof and re-obtained Smith-Tataru's conclusion. See also Zhou-Lei \cite{Zhou2} for three dimensional radially symmetric global solutions in low regularity. For higher spatial dimension ($n\geq 3$), we refer to locally well-posed in $H^s$ with $s>n/2+2/3$ by Tataru \cite{tataru2}, and recent work of Wang \cite{cbwang} on sharp local well-posedness for spherically symmetric quasilinear wave equations. In \cite{lindblad17}, Ettinger-Lindblad generalized the above results to the Einstein's equations and constructed a sharp counterexample for local well-posedness of Einstein vacuum equations in wave coordinates. An exploration by Granowski \cite{Ross} later showed that Lindblad's ill-posedness in \cite{Lind98} is stable under perturbations out of planar symmetry. In \cite{an,an2,an suv}, we generalized Lindblad's work on a scalar wave equation in \cite{Lind98} by showing that the Cauchy problems for 3D elastic waves and 3D ideal compressible MHD system are ill-posed in $H^3(\mathbb{R}^3)$ and $H^2(\mathbb{R}^3)$, respectively. In the absence of magnetic fields, the $H^2$ ill-posedness in \cite{an2} also holds for 3D compressible Euler equations. According to the low-regularity local well-posedness result in \cite{Wang19} by Wang \footnote{See also \cite{Disconzi} by Disconzi-Luo-Mazzone-Speck and \cite{zhang-andersson} by Zhang-Andersson.}, this $H^2$ ill-posedness is sharp with respect to the regularity of the fluid velocity $u$ and density $\varrho$. Unlike the 3D case, in 2D, as suggested by Smith-Tataru \cite{tataru}, the desired sharp result would be in the fractional Sobolev space $H^\frac{7}{4}(\mathbb{R}^2)$. Recently, Ohlmann \cite{Ohl} generalized Lindblad's result \cite{Lind98} to the 2D case and proved the ill-posedness for a 2D quasilinear wave equation\footnote{Ohlmann studied the model equation $\Box u=(Du)D^2 u$, with $D=\partial_{x_1}-\partial_t$.} in the logarithmic Sobolev space $H^{\frac{7}{4}}(\ln H)^{-\beta}$ with $\beta>\frac12$. This space is slightly more singular\footnote{There is a logarithmic loss between Ohlmann's result and the desired $H^{\frac{7}{4}}$ ill-posedness.} than $H^{\frac{7}{4}}$. For 2D Euler equations, the low-regularity local well-posedness was proved by Zhang \cite{zhang} in $H^s$ with $s>\frac{7}{4}$ for initial fluid velocity, vorticity and density. In this paper, we will derive the $H^{\frac{7}{4}-}$ ill-posedness and anisotropic $H^\frac74$ ill-posedness for 2D compressible Euler equations and for the more general 2D ideal MHD system. We employ the wave-decomposition approach introduced by John \cite{john74} to decompose our system. Compared with the 3D case, the decomposed system here is equipped with a different structure, which is owing to the correct selection of the eigenvectors for the coefficient matrix in 2D. Using the initial-data construction adopted in this paper, in \cite{an3} we also report the $H^{\frac{11}{4}}$ ill-posedness for the 2D elastic wave system.

For the incompressible case, the corresponding low-regularity ill-posedness in Sobolev spaces was first proved by Bourgain-Li \cite{bourgain-li3} in 2D and \cite{bourgain-li} in 3D. We also refer to Elgindi-Jeong \cite{elgindi} for a different proof in 2D, and refer to Bourgain-Li \cite{bourgain-li2} for ill-posedness in $C^m$ spaces.

As mentioned above, our ill-posedness is driven by the instantaneous shock formation. The studies on shock formation for Euler equations and other hyperbolic systems have a long history. In 1D case, the classical method is to analyze along the characteristics. For compressible Euler equations, Riemann \cite{riemann} proved the finite-time shock formation via using the so-called Riemann invariants. This was later extended to $2\times 2$ genuinely nonlinear strictly hyperbolic system by Lax \cite{lax}. For larger $n \times n$ hyperbolic systems, John \cite{john74} developed the decomposition--of--waves method to prove shock formation for genuinely nonlinear strictly hyperbolic system. See also \cite{liu} by Liu and \cite{Zhou} by Zhou for generalizations of John's approach to linearly degenerate hyperbolic systems. In \cite{christodoulou}, Christodoulou-Perez also studied the shock formation for electromagnetic plane waves in nonlinear crystals. The equations of the electromagnetic plane waves satisfy a first-order strictly hyperbolic and genuinely nonlinear system. In more than one spatial dimensions, with no symmetry assumption, Alinhac \cite{Alinhac99,Alinhac99II} proved singularity formation for solutions to quasilinear wave equations via a Nash-Moser iteration scheme. Based on this approach, Yin-Zheng-Jin \cite{yin} constructed the blow-up solution for 2D irrotational compressible Euler equations. This approach does not reveal information beyond the first blow-up point. In \cite{christodoulou10}, Christodoulou developed a geometric approach and provided a detailed understanding and a complete description of shock formation for 3D irrotational compressible Euler equations. This work was later extended to a large class of equations, see \cite{speckbk,Speck16,Speck18,Speck-luk,Speck-luk2,miao,christodoulou-miao}. In particular, applying Christodoulou's geometric approach, Luk-Speck\cite{Speck-luk,Speck-luk2} constructed shock formation for 2D and 3D compressible Euler equations, allowing the presence of non-trivial vorticity. Buckmaster-Shkoller-Vicol \cite{buckmaster1,buckmaster} also gave a different approach via using self-similar variables. For 2D compressible MHD, very little is known about its shock formation and the corresponding low-regularity solutions. For its small data global well-posedness results, we refer to Hu \cite{hu}, Wu-Wu \cite{wu} and references therein. We also refer to Wu-Zhu \cite{wu2} for global well-posedness result in a periodic domain.

\subsection{Main results}
We construct counterexamples to low-regularity $H^s$ (with $s\leq\frac74$) local well-posedness for the 2D ideal compressible MHD system \eqref{MHD}. The ill-posedness is driven by the instantaneous shock formation. In particular, we establish the shock formation in a planar symmetric region. This means we focus on a planar symmetric region to show shock formation, while the initial data remain non-symmetric and compactly supported in the entire space. Under plane symmetry, by Gauss's law, the first component $H_1$ of the magnetic field remains constant. Denoting this constant by $\kappa$, we proceed with the $\kappa\neq0$ case and the $\kappa=0$ case, respectively. In the former case, the system is strictly hyperbolic. While for $\kappa=0$, it is not\footnote{For $\kappa=0$ case, see Section \ref{h10} for the details.}. Besides $H_1$, the remaining unknowns $\Phi=(u_1,u_2,\varrho-1, H_2,S)^T$ then satisfy a quasilinear hyperbolic system stated in \eqref{1order}. Our main theorem is stated as below.
\begin{thm} \label{2D}
We say that the Cauchy problem of the 2D ideal compressible MHD equations \eqref{MHD} is ill-posed in $H^s(\mathbb{R}^2)$ if:
There exists a family of compactly supported, smooth initial data satisfying
\begin{equation} \label{zero data}
{\Big\|\big(H_{1}^{(\eta)}(x_1,x_2,0)-\kappa,\Phi^{(\eta)}(x_1,x_2,0)\big)\Big\|}_{{{H}}^s(\mathbb{R}^2)} \to 0\quad \text{as} \quad \eta\to0
\end{equation}
with $\kappa$ denoting the constant for equilibrium state, and $\eta>0$ being a small parameter. For each initial datum, there exists finite $T_\eta^*>0$ such that the corresponding Cauchy problem of the 2D ideal MHD system \eqref{MHD} admits a local-in-time regular solution $(H_1^{(\eta)},\Phi^{(\eta)})\in C^\infty\big(\mathcal{D}(\Omega_0;T_\eta^*)\big)$, and blow up in $H^s$ at $T_\eta^*$. Here, $\Omega_0$ (defined in \eqref{init domain}) denotes the spatial region where the initial data are planar symmetric, and $\mathcal{D}(\Omega_0;T_\eta^*)$ denotes the domain of dependence between $[0,T_\eta^*)$, in which the solution is purely determined by initial data in $\Omega_0$.

In this sense, we establish both the $H^{\frac74-}$ ill-posedness and the anisotropic\footnote{Specifically, the Cauchy problem is ill-posed in the anisotropic Sobolev space $H^{\frac74-\epsilon,\epsilon}$ for any small parameter $\epsilon>0$. See Section \ref{data} for the definition of anisotropic Sobolev space.} $H^\frac74$ ill-posedness to the local well-posedness of the 2D ideal compressible MHD system with the following statements$H^{\frac74}$:
\begin{itemize}
\item[\textup{i)}]{\bf(Instantaneous shock formation)} For each solution $\Phi^{(\eta)}$, a shock forms at $T_\eta^*$. More precisely, we have that $|\Phi^{(\eta)}|$ remains small, while $|\nabla \Phi^{(\eta)}|$ blows up at $T_\eta^*$. Moreover, the lifespan vanishes as $\eta \to 0$, i.e., $T_\eta^* \to 0$, which means that the constructed initial data $\Phi_0^{(\eta)}$ give rise to the instantaneous shock formation.

\item[\textup{ii)}]{\bf(Blow-up of $H^1$-norm)} The $H^1$-norm of the solution $\Phi^{(\eta)}$ blows up at $T^*_{\eta}$. In particular, we have
\begin{equation*} 
      \lim\limits_{t\to T_\eta^{*}}\|\nabla u_1^{(\eta)}(\cdot,t)\|_{L^2(\Omega_{t})}=+\infty,\quad     \lim\limits_{t\to T_\eta^{*}}\|\nabla \varrho^{(\eta)}(\cdot,t)\|_{L^2(\Omega_{t})}=+\infty,
\end{equation*}
where $\Omega_{t}$ denotes the $t$-slice of $\mathcal{D}(\Omega_0;T_\eta^*)$.
\end{itemize}

For the isotropic $H^\frac74$ case, we have the following ill-posedness: there exists a family of initial data whose $H^\frac74$ norms shrink to $0$ as $\eta\to0$, yet the (first) shock of the corresponding solutions form in a future time $T_\eta^*$ for every $\eta$.\footnote{See Remark \ref{weak ill-posedness} for the details.}
\end{thm}

\begin{remark}
In our proof, under plane symmetry, we algebraically reformulate the MHD system \eqref{MHD} via a wave-decomposition approach. The success of our analysis crucially relies on a key fact that all the coefficients of the decomposed hyperbolic system are uniformly bounded. For a real physical system without any imposed mathematical condition on its structure, it is far from trivial that the coefficients $c^i_{im}$, $\gamma^i_{im}$, $\gamma^i_{km}$ of the decomposed system defined in \eqref{coec}-\eqref{coeg2} are bounded. Our carefully designed eigenvectors \eqref{regv}-\eqref{legv} ensure the required boundedness. Instead of directly verifying around 300 coefficients as for 3D in \cite{an2}, in this paper, for the 2D case, we find a concise proof via analyzing the leading-order terms of the potential problematic coefficients. See Section \ref{sbd} for the discussion.
\end{remark}
\begin{remark} \label{rk data}
To construct suitable initial data, we get inspired by and extend the ideas in Lindblad \cite{Lind96,Lind98}, Ohlmann \cite{Ohl} and An-Chen-Yin \cite{an2}. In our ansatz, the initial data take the form
\begin{equation*}
w_1^{(\eta)}(x_1,x_2,0)=-\theta |\ln (x_1)|^\alpha \mathcal{X}(\frac{x_1}{\eta})\psi\Big(\frac{|\ln(x_1)|^\delta x_2}{\sqrt{x_1}}\Big).
\end{equation*}
where $\theta,\alpha,\eta,\delta$ are constant parameters, and $\mathcal{X},\psi$ are certain functions of their arguments. Here $w_1^{(\eta)}$ is a certain combination of $\Phi_{x_1}=(\partial_{x_1}u_1,\partial_{x_1}u_2,\partial_{x_1}\varrho, \partial_{x_1}H_2,\partial_{x_1}S)^T$. Using these functions and the finite propagation speed argument in Appendix \ref{FSP}, we (spatially) localize the domain of dependence to a neighborhood of $x_2=0$, and conduct the shock formation argument within the planar symmetric portion. Moreover, to establish ill-posedness at the threshold regularity $H^\frac34$, we carefully analyze the relationship among these parameters, and design the cut-off functions $\mathcal{X},\psi$. We refer readers to Section \ref{data} for the detailed construction. 
\end{remark}
\begin{remark}
   Building on the ill-posed data of $w_1^{(\eta)}$, we can further construct counterexamples to local existence of low regularity solutions by specifying a rapidly decaying sequence of $\{\eta_m\}$ and adding the data of $w_1^{(\eta_m)}$, the resulting sum remains uniformly bounded in the Sobolev space, but no solution exists for any $T>0$ (because $T_{\eta_m}^* \to 0$). See Section \ref{data} for details. 
\end{remark}
\begin{remark}
    In order to find suitable $H_1$ to fulfill the constraint $\nabla\cdot H=0$, especially in the non-symmetric region, we solve the following transport equation at $t=0$ for each fixed $x_2$ : 
    $$\partial_{x_1} H_1(x_1,x_2,0)=-\partial_{x_2}H_2(x_1,x_2,0).$$
   Here, we set the initial data of $H_2$ to be compactly supported, and odd with respect to $x_1$. With $H_2$ given, we solve the above equation for the initial data of $H_1-\kappa$, which are compactly supported within a small domain. For further details, see Section \ref{data}.
\end{remark}
\begin{remark}
Besides the just-mentioned constructed initial data, compared with Ohlmann's approach in \cite{Ohl}, our result also exhibits the following differences. First, we investigate the \underline{physical} MHD system \eqref{MHD} and the Euler equations \eqref{euler1}. The specific structures of these two systems are \underline{crucial} for our analysis. Ohlmann \cite{Ohl} studied a quasilinear equation
\begin{equation*}
  \Box u=(Du)D^2 u, \quad \text{with} \ D=\partial_{x_1}-\partial_t.
\end{equation*}
Second, in \cite{Ohl} by Ohlmann, there is a logarithmic loss for the regularity, i.e., the ill-posedness in \cite{Ohl} holds in the logarithmic Sobolev space $H^{\frac{7}{4}}(\ln H)^{-\beta}$ with $\beta>\frac12$. For the 2D ideal MHD system and the Euler equations, we provide a new proof of the estimate for Sobolev norm by using interpolation theorem in the physical space, and prove the $H^s$ ill-posedness for $s<\frac74$ (isotropic ill-posedness) and the anisotropic $H^\frac74$ ill-posedness in the threshold regularity. Third, in our work, the underlying mechanism behind our ill-posedness is explored. And we demonstrate that the ill-posedness results from the instantaneous shock formation.
\end{remark}
As a corollary of the above theorem, in the absence of magnetic fields, i.e., $H_1=H_2\equiv0$, we have the following ill-posedness result for the 2D compressible Euler equations \eqref{euler1} allowing non-trivial entropy and vorticity.
\begin{thm} \label{sharpeuler}
For the 2D compressible Euler equations \eqref{euler1}, the Cauchy problem is ill-posed in $H^{\frac74-}$ and anisotropically ill-posed in $H^\frac74$, with respect to the fluid velocity and density. More precisely, there exits a family of compactly supported smooth initial data for \eqref{euler1} satisfying
$${\|\varrho_0^{(\eta)}\|}_{X}+{\|u_0^{(\eta)}\|}_{X}+{\|S_0^{(\eta)}\|}_{X} \to 0\quad \text{as} \quad \eta\to0,$$
where $X={H}^{\frac74-}(\mathbb{R}^2)$ or ${H}^{\frac74-\epsilon,\epsilon}(\mathbb{R}^2)$ (see Definition \ref{aniso space}), and the subscript ``0" denotes the initial data. After a finite time $T_\eta^*$, a singularity forms in $\mathcal{D}(\Omega_0;T_\eta^*)$ and the solution ceases to be smooth. Tied to this family of initial data, this singularity formation is instantaneous, i.e., $T_\eta^* \to 0$ as $\eta \to 0$. Furthermore, at $T^*_\eta>0$, the singularity corresponds to the (first) shock in planar symmetric region. Moreover, the $H^1$-norm of the velocity $u_1$ and density $\varrho$ blow up at the shock formation time $T^*_{\eta}$:
\begin{equation*}
 \lim\limits_{t\to T_\eta^{*}} \|u_{1}^{(\eta)}(\cdot,t)\|_{H^1(\Omega_t)}=+\infty,\quad \lim\limits_{t\to T_\eta^{*}}\|\varrho^{(\eta)}(\cdot,t)\|_{H^1(\Omega_t)}=+\infty.
\end{equation*}
\end{thm}
\begin{remark} \label{rk euler}
For the 2D compressible Euler equations, as is well-known, the classical local well-posedness result holds for initial data in $H^s$ with $s>2$. Recently, Zhang \cite{zhang} improved the classical result by proving the low regularity local well-posedness in $H^s$ space with $s>\frac{7}{4}$ for fluid velocity, vorticity and density. Zhang's result still leave room for improvement regarding the regularity of fluid vorticity. See the following two pictures:
\begin{figure}[H]
\centering
\begin{minipage}[t]{0.48\textwidth}
\centering
\begin{tikzpicture}[scale=0.9, fill opacity=0.5, draw opacity=1, text opacity=1]
\filldraw[white, fill=gray!40](1.5,1)--(1.5,4)--(5,4)--(5,1)--(1.5,1);
\filldraw[gray!40, fill=gray!80](1,1)--(1,4)--(1.5,4)--(1.5,1)--(1,1);
\draw[->](0,-0.3)--(5.5,-0.3);
\filldraw(5,-0.3)node[below]{\footnotesize $\tilde{s}$: regularity for the density};
\draw[->](0,-0.3)--(0,4)node[right,above]{\footnotesize $s'$: regularity for the vorticity};
\draw[dashed](1,0.5)--(5,0.5);
\draw[dashed](1,1)--(5,1);
\draw[dashed](5,0.5)--(5,4);
\draw[dashed](1.5,1)--(1.5,4);
\draw[dashed](1,4)--(5,4);
\filldraw [thick,black] (1,0.5) circle [radius=0.8pt];
\filldraw [thick,black] (1,1) circle [radius=0.8pt];
\filldraw [thick,black] (1,-0.3) circle [radius=0.8pt]node[below]{$\frac74$};
\filldraw [thick,black](1.5,-0.3) circle [radius=0.8pt]node[below]{$2$};

\filldraw [thick,black](0,1) circle [radius=0.8pt]node[left]{$\frac74$};
\filldraw [thick,black](0,0.5) circle [radius=0.8pt]node[left]{$1$};
\filldraw(1.65,1.9)circle node[right]{\scriptsize LWP by \cite{zhang}};
\filldraw(1.9,3)circle node[right]{\scriptsize Classical LWP};
\filldraw (2,2) circle;
\draw[->](1.75,1.9)--(1.25,1.9);
\draw[->](-0.1,2.25)--(0.95,2.25);
\filldraw(0,2.5)node[left]{\footnotesize An-Chen-Yin's };
\filldraw(0,2.2)node[left]{\footnotesize ill-posedness};
\filldraw(0,1.9)node[left]{\footnotesize (Theorem \ref{sharpeuler})};
\draw[very thick,black](1,0.5)--(1,4);
\filldraw(2.7,0.75)circle node[right]{?};
\end{tikzpicture}
\caption{Vorticity regularity $s'$ \textit{\bf vs.} density regularity $\tilde{s}$.}
\end{minipage}
\begin{minipage}[t]{0.48\textwidth}
\centering
\begin{tikzpicture}[scale=0.9, fill opacity=0.5, draw opacity=1, text opacity=1]
\filldraw[white, fill=gray!40](1.5,0.5)--(1.5,4)--(5,4)--(1.5,0.5);
\filldraw[gray!40, fill=gray!80](1,1)--(1,4)--(1.5,4)--(1.5,1)--(1,1);
\draw[->](0,-0.3)--(5.5,-0.3);
\filldraw(5,-0.3)node[below]{\footnotesize $s$: regularity for the velocity};
\draw[->](0,-0.3)--(0,4)node[right,above]{\footnotesize $s'$: regularity for the vorticity};
\draw(1.5,0.5)--(5,4);
\draw[dashed](1.5,0.5)--(1.5,4);
\draw[dashed](1,1)--(1.5,1);
\draw[dashed](1,4)--(5,4);
\filldraw [thick,black] (1.5,0.5) circle [radius=0.8pt];
\filldraw [thick,black] (1,1) circle [radius=0.8pt];
\filldraw [thick,black] (1,0.5) circle [radius=0.8pt];
\filldraw [thick,black] (1,-0.3) circle [radius=0.8pt]node[below]{$\frac{7}{4}$};
\filldraw [thick,black](1.5,-0.3) circle [radius=0.8pt]node[below]{$2$};

\filldraw [thick,black](0,0.5) circle [radius=0.8pt]node[left]{1};
\filldraw [thick,black](0,1) circle [radius=0.8pt]node[left]{$\frac74$};

\draw[thick,black](1,0.5)--(1.5,0.5);
\draw[very thick,black](1,0.5)--(1,4);
\filldraw(1.65,1.9)circle node[right]{\scriptsize LWP by \cite{zhang}};
\draw[->](1.75,1.9)--(1.25,1.9);
\filldraw(1.65,3)circle node[right]{\scriptsize Classical LWP};
\filldraw(1.65,0.7)circle node[right]{\scriptsize Ill-posedness by \cite{bourgain-li3}\cite{elgindi} };
\filldraw(1.05,0.75)circle node[right]{?};
\filldraw (2,2) circle;
\draw[->](1.75,0.3)--(1.25,0.45);
\filldraw(1.65,0.35)node[right]{\scriptsize (incompressible 2D Euler)};
\draw[->](-0.1,2.25)--(0.95,2.25);
\filldraw(0,2.5)node[left]{\footnotesize An-Chen-Yin's };
\filldraw(0,2.2)node[left]{\footnotesize ill-posedness};
\filldraw(0,1.9)node[left]{\footnotesize (Theorem \ref{sharpeuler})};
\end{tikzpicture}
\caption{Vorticity regularity $s'$ \textit{\bf vs.} velocity regularity $s$.}
\end{minipage}
\end{figure}
\noindent Note that in the above figures, when the vorticity regularity $s'$ is fixed, the regularity of velocity can be at most $s'+1$. In contrast, the density regularity $\tilde{s}$ is not necessarily related to $s'$. Our ill-posedness result demonstrates that the threshold regularity is $H^\frac74$ (with respect to the fluid velocity and density). 
\end{remark}
\subsection{Organization of the paper}
As in \cite{an2,Lind96,Lind98}, we first investigate the shock formation under plane symmetry. Suppose that $U(x_1,t)=(u_1,u_2,\varrho,H_1,H_2,S)(x_1,t)$ is a planar symmetric solution to \eqref{MHD}. Then, by Gauss's Law, we have that $H_1\equiv \kappa$. The ideal MHD system \eqref{MHD} is then reduced to a $5\times 5$ first-order hyperbolic system:
\begin{equation*}
        \partial_t\Phi+A(\Phi)\partial_{x_1}\Phi=0
\end{equation*}
with $\Phi=(u_1,u_2,\varrho-1,H_2,S)^T$. Then we algebraically decompose the system via employing John's wave-decomposition method in \cite{john74}. The decomposed system satisfies \eqref{eqrho}-\eqref{eqv}, which forms a transport system along different characteristics. In Section \ref{pre}, we carefully explore the structures of the corresponding system.

In Section \ref{data}, we construct vanishing $H^s$ initial data (for $s\leq\frac74$) to establish the ill-posedness for the Cauchy problems of \eqref{MHD}. Our construction of initial data here is inspired by Lindblad \cite{Lind96,Lind98}, Ohlmann \cite{Ohl} and An-Chen-Yin \cite{an2}. We further modify and extend the construction in \cite{an2,Lind96,Lind98,Ohl} by carefully analyzing the relationships among parameters and the regularity along different directions.

We first consider the $\kappa\neq0$ case. Based on the decomposition of waves, in Section \ref{aprior} we prove the (first) shock formation at time $T_\eta^*<+\infty$. And we exhibit the detailed quantitative information all the way up to the shock formation time $T_\eta^*$. The key point here is to trace the evolution of $\rho_i$ ($i=1,\cdots,5$), the so-called inverse density of the $i^{\text{th}}$ characteristics. We derive a positive uniform lower bound for $\{\rho_i\}_{i=2,\cdots,5}$. Whereas, for the first family of characteristics, its inverse density $\rho_1(z_0,t)\to 0$ as $t\to T_\eta^*$, which indicates the compression of characteristics and whence the shock formation.

Prescribing the aforementioned initial data, we have $T_\eta^*\to 0$ as $\eta \to 0$. In Section \ref{ill}, we then estimate the $H^1(\mathbb{R}^2)$ norm of the solution to \eqref{MHD} at time $T_\eta^*$. In a suitable constructed spatial region $\Omega_{T^*_\eta}$, we deduce that $\|\Phi(\cdot,T^*_\eta)\|_{H^1(\Omega_{T^*_\eta})}^2=+\infty$. This further implies that the $H^\frac{7}{4}$ norm also blows up within $\Omega_{T^*_\eta}$, and the desired ill-posedness result. Furthermore, both the inflation of $H^1$ norm and the vanishing of $T_\eta^*$ are driven by the instantaneous shock formation: $\rho_1(z_0,t)\to 0$ as $t\to T_\eta^*.$

In a similar manner, we treat the non-strictly hyperbolic $\kappa=0$ case. The details are in Section \ref{h10}. Finally, in Section \ref{euler ill}, we apply the above method and results for the 2D compressible Euler equations and obtain the aforementioned ill-posedness results.

In Appendix \ref{FSP}, we show that the 2D ideal MHD system (without symmetry condition) obeys finite speed of propagation.

\subsection{Acknowledgements}
XA is supported by MOE Tier 1 grants A-0004287-00-00, A-0008492-00-00 and MOE Tier 2 grant A-8000977-00-00. HC acknowledges the support of MOE Tier 2 grant A-8000977-00-00 and NSFC (Grant No. 12171097). SY is supported by NSFC under Grant No. 12001149.

\section{Preliminaries} \label{pre}
In this section, we apply the wave-decomposition method to the 2D ideal compressible MHD system \eqref{MHD} and analyze its structure.

\subsection{The corresponding hyperbolic system}
We establish the ill-posedness under planar symmetry. The planar symmetric solution of \eqref{MHD} satisfies $$U(x_1,x_2,t)=(u_1,u_2,\varrho,H_1,H_2,S)(x_1,x_2,t)=(u_1,u_2,\varrho,H_1,H_2,S)(x_1,t).$$
 Using Gauss's Law, i.e., the last equation in \eqref{MHD}, we have
\begin{equation*}
\nabla\cdot H=\frac{\partial H_1}{\partial x_1}+\frac{\partial H_2}{\partial x_2}=\frac{\partial H_1}{\partial x_1}=0, \quad \text{hence}\,\, H_1(x_1,t)=H_1(t).
\end{equation*}
Moreover, following from the third equation of \eqref{MHD}, it holds that $\partial_t H_1=0.$ Thus, $H_1(x_1,t)$ remains to be a constant $\kappa$ during the evolution. For notational simplicity, we denote $x_1$ by $x$ hereafter. In this paper, we assume that the constant $\kappa$ is sufficiently small such that
\begin{equation} \label{h1}
\kappa^2 \ll \min\{\mu_0^{-1}A\gamma,1\},
\end{equation}
where $\mu_0$ denotes the magnetic permeability constant, $A$ is a positive constant and $\gamma>1$ is the adiabatic gas constant. Via \eqref{MHD}, the remaining unknowns $(u_1,u_2,\varrho,H_2,S)^T$ then satisfy the following system:
\begin{equation}\label{eq}
\left\{\begin{split}
&\partial_t u_1+\ u_1\partial_{x} u_1+\frac{c^2}{\varrho} \partial_{x}\varrho+\frac{\mu_0H_2}{\varrho}\partial_{x}H_2+\frac{c^2}{\gamma} \partial_{x} S=0,\\
&\partial_t u_2+ u_1\partial_{x} u_2-\frac{\mu_0H_1}{\varrho}\partial_{x_1}H_2=0,\\
&\partial_t\varrho+\varrho\partial_{x} u_1+u_1\partial_{x}\varrho=0,\\
&\partial_t H_2+ u_1\partial_{x} H_2+H_2\partial_xu_1-H_1\partial_xu_2=0,\\
&\partial_t S+u_1\partial_xS=0,
\end{split}\right.
\end{equation}
where $c=\sqrt{\partial_\varrho p}$ is the sound speed. Let $\Phi=(u_1,u_2,\varrho-1, H_2,S)^T$. The quasilinear hyperbolic system \eqref{eq} then takes the form below
\begin{equation}\label{1order}
        \partial_t\Phi+A(\Phi)\partial_x\Phi=0,
\end{equation}
where
\begin{equation*} 
A(\Phi)=\begin{pmatrix}
u_1 & 0 & \frac{c^2}{\varrho} & \frac{\mu_0H_2}{\varrho} &  \frac{c^2}{\gamma} \\
0 & u_1 & 0 & -\frac{\mu_0H_1}{\varrho} & 0 \\
\varrho & 0 & u_1 & 0  & 0 \\
H_2 & -H_1 & 0 &u_1 & 0\\
0 & 0  & 0 & 0 &  u_1
\end{pmatrix}.
\end{equation*}
Via calculations, we have that the eigenvalues of the coefficient matrix $A(\Phi)$ are
\begin{equation}\label{egvl}
\begin{split}
\lambda_1=u_1+C_f,\quad \lambda_2=u_1+C_s,\quad \lambda_3=u_1,\quad
\lambda_4=u_1-C_s,\quad  \lambda_5=u_1-C_f
\end{split}
\end{equation}
with
\begin{equation*} 
\begin{split}
&C_f=\Big\{\frac{\mu_0}{2\varrho}(H_1^2+H_2^2)+\frac{c^2}{2}+\frac12\sqrt{[\frac{\mu_0}{\varrho}(H_1^2+H_2^2)+c^2]^2-\frac{4\mu_0}{\varrho}H_1^2c^2}\Big\}^{1/2},\\
&C_s=\Big\{\frac{\mu_0}{2\varrho}(H_1^2+H_2^2)+\frac{c^2}{2}-\frac12\sqrt{[\frac{\mu_0}{\varrho}(H_1^2+H_2^2)+c^2]^2-\frac{4\mu_0}{\varrho}H_1^2c^2}\Big\}^{1/2}.
\end{split}
\end{equation*}
A quasilinear hyperbolic system is called strictly hyperbolic if the eigenvalues of its coefficient matrix are all real and distinct. We first consider the case of $H_1=\kappa\neq0$. We remark that $\kappa$ is a fixed non-zero constant, which is small relative to the physical constant in \eqref{h1}. However, compared with the small parameters to appear in subsequent analysis of shock formation and ill-posedness, it is treated as an $O(1)$ constant. Here, $O(\cdot)$ denotes the Landau notation defined on the whole time interval $[0,T)$, meaning ``$\leq C\cdot$" for some uniform constant $C>0$ independent of $T$. For this case, provided that $|\Phi|<2\delta$ with $\delta$ being sufficiently small, the ideal MHD system \eqref{1order} is strictly hyperbolic, namely the above five eigenvalues are completely distinct,
\begin{align*}
    \lambda_5(\Phi)< \lambda_4(\Phi)<\lambda_3(\Phi)<\lambda_2(\Phi)<\lambda_1(\Phi).
\end{align*}
We then choose the right eigenvectors of $A(\Phi)$ as follows:
\begin{equation}\label{regv}
\begin{split}
&r_1=\left(\begin{array}{cc}
1\\
-\frac{(C_f^2 -c^2)H_1}{C_f^2H_2}\\
\frac{\varrho }{C_f}\\
\frac{\varrho(C_f^2 -c^2)}{\mu_0C_fH_2}\\
0
\end{array}\right),
\ r_2=\left(\begin{array}{cc}
H_2\\
-\frac{(C_s^2 -c^2)H_1}{C_s^2}\\
\frac{\varrho H_2}{C_s}\\
\frac{\varrho(C_s^2 -c^2)}{\mu_0C_s}\\
0
\end{array}\right),
\ r_3=\left(\begin{array}{cc} 0\\ 0\\ -\frac{\varrho}{\gamma}\\0\\1\end{array}\right), \\
& r_4=\left(\begin{array}{cc}
H_2\\
-\frac{(C_s^2 -c^2)H_1}{C_s^2}\\
-\frac{\varrho H_2}{C_s}\\
-\frac{\varrho(C_s^2 -c^2)}{\mu_0C_s}\\
0
\end{array}\right),
\ r_5=\left(\begin{array}{cc}
1\\
-\frac{(C_f^2 -c^2)H_1}{C_f^2H_2}\\
-\frac{\varrho }{C_f}\\
-\frac{\varrho(C_f^2 -c^2)}{\mu_0C_fH_2}\\
0
\end{array}\right).
\end{split}
\end{equation}
The corresponding left eigenvectors are set to be dual to the right ones satisfying:
\begin{equation}\label{lr}
l_ir_j=\delta_{ij}\quad \text{for}\quad i,j=1,\cdots,5.
\end{equation}
Here $\delta_{ij}$ is the Kronecker symbol. Specifically, we list these $l_i$:
\begin{equation} \label{legv}
\begin{split}
&l_1=\frac{C_f^4H_2}{2H_2^2C_f^4+2H_1^2(C_f^2-c^2)^2}\Big(H_2,-\frac{(C_f^2-c^2)H_1}{C_f^2},\frac{c^2H_2}{\varrho C_f}, \frac{C_f^2-c^2}{C_f },\frac{c^2H_2}{\gamma C_f}\Big),\\
&l_2=\frac{C_s^4H_2}{2H_2^2C_s^4+2H_1^2(C_s^2-c^2)^2}\Big(1,-\frac{(C_s^2-c^2)H_1}{C_s^2H_2},\frac{c^2}{\varrho C_s}, \frac{C_s^2-c^2}{C_s H_2},\frac{c^2}{\gamma C_s}\Big),\\
&l_3=(0,0,0,0,1),\\
&l_4=\frac{C_s^4H_2}{2H_2^2C_s^4+2H_1^2(C_s^2-c^2)^2}\Big(1,-\frac{(C_s^2-c^2)H_1}{C_s^2H_2},-\frac{c^2}{\varrho C_s}, -\frac{C_s^2-c^2}{C_s H_2},-\frac{c^2}{\gamma C_s}\Big),\\
&l_5=\frac{C_f^4H_2}{2H_2^2C_f^4+2H_1^2(C_f^2-c^2)^2}\Big(H_2,-\frac{(C_f^2-c^2)H_1}{C_f^2},-\frac{c^2H_2}{\varrho C_f}, -\frac{C_f^2-c^2}{C_f},-\frac{c^2H_2}{\gamma C_f}\Big).\\
\end{split}
\end{equation}
\begin{rmk}
Designing proper right and left eigenvectors is crucial in our approach. As we will explain below, the above selection of eigenvectors guarantees that the equivalent form of \eqref{MHD} satisfies the genuinely nonlinear condition, as well as the boundedness condition for the corresponding coefficients in the decomposed system.
\end{rmk}
\subsection{Decomposition of waves}\label{decom}
Employing the left and right eigenvectors designed in the above subsection, we introduce
\begin{equation}\label{wi}
  w_i:=l_i\partial_x\Phi \quad \text{for} \quad i=1,\cdots,5.
  \end{equation}
By \eqref{lr} and \eqref{wi}, there holds the following formula
\begin{equation}\label{phix}
  \partial_x\Phi=\sum_{k=1}^5 w_kr_k.
\end{equation}
Before writing down the explicit equations for $w_i$, we first introduce the characteristic and bi-characteristic coordinates. We define the {\bf $i^{\text{th}}$ characteristic} originating from $z_i$ to be $\mathcal{C}_i(z_i)=\Big\{\Big(X_i(z_i,t),t\Big): 0\leq t\leq T\Big\}$ with $X_i(z_i,t)$ solving
\begin{align}\label{flow}
  \left\{\begin{array}{ll}
  \frac{\partial}{\partial t}X_i(z_i,t)=\lambda_i\big(\Phi(X_i(z_i,t),t)\big),\quad t\in[0,T],\\
  X_i(z_i,0)=z_i.
  \end{array}\right.
\end{align}
For any $(x,t)\in \mathbb{R}\times [0,T]$, there is a unique $(z_i,s_i)\in \mathbb{R}\times [0,T]$ such that $(x,t)=\big(X_i(z_i,s_i),s_i\big)$. Hence, we call $(z_i,s_i)$ {\bf the characteristic coordinate}. In addition, we can also locate the points using the intersection of two transversal characteristics $\mathcal{C}_i(y_i)$ and $\mathcal{C}_j(y_j)$ when $i\neq j$. For any $(x,t)\in \mathbb{R}\times [0,T]$, there is a unique pair of $(y_i,y_j)$ such that the characteristics $\mathcal{C}_i$ and $\mathcal{C}_j$ intersect at $(x,t)$. We define $(y_i,y_j)$ to be {\bf\textit{the bi-characteristic coordinate}}. In particular, these coordinates satisfy
\begin{equation*}
  (x,t)=\big(X_i(y_i,t'(y_i,y_j)),t'(y_i,y_j)\big)=\big(X_j(y_j,t'(y_i,y_j)),t'(y_i,y_j)\big).
\end{equation*}
A direct calculation yields the following rules of coordinate transformations:
\begin{equation}
  \partial_{z_i}=\rho_i\partial_x,\quad \partial_{s_i}=\lambda_i\partial_x+\partial_t,
\quad
  \partial_{y_i}t'=\frac{\rho_i}{\lambda_j-\lambda_i},\quad\partial_{y_j}t'=\frac{\rho_j}{\lambda_i-\lambda_j},
\end{equation}
\begin{equation}\label{biytoz}
  \partial_{y_i}=\frac{\rho_i}{\lambda_j-\lambda_i}\partial_{s_j}=\partial_{z_i}+\frac{\rho_i}{\lambda_j-\lambda_i}\partial_{s_i}
\end{equation}
and
\begin{equation}
  dx=\rho_idz_i+\lambda_ids_i,\quad dt=ds_i,\quad   dz_i=dy_i,\quad dz_j=dy_j,
\end{equation}
\begin{equation}
  dx=\frac{\rho_i\lambda_j}{\lambda_j-\lambda_i}dy_i+\frac{\rho_j\lambda_i}{\lambda_i-\lambda_j}dy_j,\quad dt=\frac{\rho_i}{\lambda_j-\lambda_i}dy_i+\frac{\rho_j}{\lambda_i-\lambda_j}dy_j,
\end{equation}
where we employ a geometric quantity $\rho_i$ (the inverse density of the $i^{\text{th}}$ characteristics) to characterize the shock formation,
\begin{equation*}
\rho_i:=\partial_{z_i}X_i.
\end{equation*}
This geometric quantity describes the compression among the characteristics of the $i^{\text{th}}$ family. And it follows from \eqref{flow} that
\begin{equation}\label{319}
  \rho_i(z_i,0)=1.
\end{equation}
We further define
\begin{equation} \label{defvi}
v_i:=\rho_iw_i.
\end{equation}
Then, via the wave decomposition as we did in \cite{an2}, we get the following system for $(\rho_i,w_i,v_i)$:
\begin{align}
  \partial_{s_i}\rho_i=&c_{ii}^iv_i+\Big(\sum_{m\neq i}c_{im}^iw_m\Big)\rho_i,\label{eqrho}\\
  \partial_{s_i}w_i=&-c_{ii}^iw_i^2+\Big(\sum_{m\neq i}(-c_{im}^i+\gamma_{im}^i)w_m\Big)w_i+\sum_{m\neq i,k\neq i\atop m\neq k}\gamma_{km}^iw_kw_m,\label{eqw}\\
  \partial_{s_i}v_i=&\Big(\sum_{m\neq i}\gamma_{im}^iw_m\Big)v_i+\sum_{m\neq i,k\neq i\atop m\neq k}\gamma_{km}^iw_kw_m\rho_i,\label{eqv}
\end{align}
where $\partial_{s_i}=\lambda_i\partial_x+\partial_t$ and
 \begin{align}
&c_{im}^i=\nabla_\Phi\lambda_i\cdot r_m,\label{coec}\\
  &\gamma_{im}^i=-(\lambda_i-\lambda_m)l_i \cdot(\nabla_\Phi r_i \cdot r_m-\nabla_\Phi r_m \cdot r_i),\quad \text{when} \  m\neq i,\label{coeg1}\\
  &\gamma_{km}^i=-(\lambda_k-\lambda_m)l_i \cdot (\nabla_\Phi r_k \cdot r_m), \qquad\qquad\qquad \text{when} \  k\neq i,\  m\neq i.\label{coeg2}
\end{align}
Noting that in \eqref{eqw}, if $c_{ii}^i\neq 0$ always holds for certain $i$, then, the $i^{\text{th}}$ characteristics are {\bf genuinely nonlinear} in the sense of P. D. Lax and $w_i$ obeys a Ricatti-type ODE. Notice that along the first characteristics, in \eqref{eqrho} and \eqref{eqw}, we have $c_{11}^1(\Phi)$ being away from zero. This is owing to
 \begin{equation*}
 \begin{split}
 &c^1_{11}(0)\Big|_{H_1=0}=\nabla_\Phi\lambda_1(0)r_1(0)\\
 =&(1,0,\frac{\sqrt{A\gamma}(\gamma-1)}{2},0,\frac{\sqrt{A\gamma}}{2})\cdot(1,0,\frac{1}{\sqrt{A\gamma}},0,0)^T\\
 =&\frac{\gamma+1}{2}>0.
 \end{split}\end{equation*}
Regarding $H_1$ as a parameter, when $H_1$ and $|\Phi|$ are small enough, we have that $c_{11}^1(\Phi)\approx\frac{\gamma+1}{2}$. This implies that the genuine nonlinearity, i.e., $c_{11}^1(\Phi)>0$.

\subsection{Boundedness of the coefficients}\label{sbd}
Consider the coefficients $c^i_{im},\gamma_{im}^i,\gamma^i_{km}$ given in \eqref{coec}-\eqref{coeg2}. Note that $C_f^2-c^2=0$ when $H_2=0$. Furthermore, it also holds $\partial_{H_2}(C_f^2-c^2)=O(H_2)$ and $\partial^2_{H_2H_2}(C_f^2-c^2)=O(1)$.  Hence, the leading order of $C_f^2-c^2$ is $O(H_2^2)$. One can then verify that $l_i$, $\nabla_\Phi r_i$ and $r_i$ are all regular provided that $|\Phi|$ is small enough. Hence, the coefficients $|c^i_{ii},\gamma^i_{im},\gamma^i_{km}|$ are all uniformly bounded from above under the assumption $\Phi\in B_{2\delta}^5(0)$.

\subsection{Characteristic strips}
At the end of this section, we introduce the characteristic strips in which our estimates would be carried out. For initial data supported in $I_0=[a,b]$, we define the {\bf$i^{\text{th}}$ characteristic strip} $\mathcal{R}_i$ by
\begin{equation*} 
\mathcal{R}_i:=\cup_{z_i\in I_0}\mathcal{C}_i(z_i).
\end{equation*}
Via deriving estimates for the decomposed system \eqref{eqrho}-\eqref{eqv}, we can prove the formation of shock, which is characterized by the vanishing of $\rho_1$ in $\mathcal{R}_1$. See Section \ref{aprior} for the details.

Corresponding to five different characteristic speeds $\lambda_i$, there are five characteristic strips starting from $I_0=[a,b]$. Now we show that these characteristic strips $\mathcal{R}_1,\cdots,\mathcal{R}_5$ totally separate with each other after a time $t_0$. We take the supremum and infimum of the eigenvalues
\begin{align*}
  \bar{\lambda}_i:=\sup_{\Phi\in B_{2\delta}^5 (0)}\lambda_i(\Phi),\quad \underline{\lambda}_i:=\inf_{\Phi\in B_{2\delta}^5(0)}\lambda_i(\Phi),\quad \text{for}\quad i=1,\cdots,5
\end{align*}
and define
\begin{equation*}
  \sigma:=\min_{\alpha<\beta\atop \alpha,\beta\in\{1,\cdots,5\}}(\underline{\lambda}_\alpha-\bar{\lambda}_\beta).
\end{equation*}
Provided $|\Phi|$ being sufficiently small, it holds that $$\sigma\approx \min_{\alpha<\beta\atop \alpha,\beta\in\{1,\cdots,5\}}(\underline{\lambda}_\alpha(0)-\bar{\lambda}_\beta(0))=O(c^2)>0.$$ Hence, $\sigma$ has a uniform positive lower bound. For $\alpha\in\{1,\cdots,5\}$ and $z\in I_0$, by \eqref{flow}, we obtain that
\begin{equation*}
  z+\underline{\lambda}_\alpha t\leq X_\alpha(z,t)\leq z+\bar{\lambda}_\alpha t.
\end{equation*}
Then, for all $\alpha<\beta$, with $\alpha,\beta\in\{1,\cdots,5\}$, we have
\begin{equation*}
   X_\alpha(a,t)-X_\beta(b,t) \geq (a+\underline{\lambda}_\alpha t)-(b+\bar{\lambda}_\beta t)=-(b-a)+(\underline{\lambda}_\alpha-\bar{\lambda}_\beta)t\geq-(b-a)+\sigma t.
\end{equation*}
Note that the above inequality is strictly positive if
\begin{equation}\label{t0}
  t>t_0:=\frac{b-a}\sigma>0.
\end{equation}
Hence the above $t_0$ in \eqref{t0} is the separating time, after which the five characteristic strips are well separated.

\section{Constructions of initial data}\label{data}
In this section, we construct the $H^\frac{7}{4}$ initial data for the 2D ideal compressible MHD system \eqref{MHD}, with which we prove the desired ill-posedness. Our construction of initial data is inspired by Lindblad \cite{Lind96,Lind98}, Ohlmann \cite{Ohl} and An-Chen-Yin \cite{an2}. In \cite{Lind96,Lind98}, Lindblad employed the function $|\ln (x)|^\alpha$ to construct initial data with low regularity. The corresponding Cauchy problems of the 3D quasilinear wave equations are proved to be ill-posed in suitable Sobolev spaces. In \cite{Ohl}, Ohlmann generalized Lindblad's result to 2D quasilinear wave equations. He introduced a cut-off function $\phi(\frac{|\ln(x)|^\delta x_2}{\sqrt{x}})$ such that the initial data belong to the logarithmic Sobolev space $H^\frac{7}{4}(\ln H)^{-\beta}$, which is slightly more singular than the fractional Sobolev space $H^\frac{7}{4}$. In this paper, we generalize Lindblad's and Ohlmann's approach and also extend our approach in \cite{an2}. Instead of working in the Fourier space as in \cite{Ohl}, we give a new proof to estimate the initial data. Here, we first estimate the $H^1$ and $H^2$ norms in the physical space. Then, we employ the interpolation theorem to derive the $H^\frac74$ norm estimate.

We first construct the $H^\frac{3}{4}$ initial data for the decomposed system \eqref{eqrho}-\eqref{eqv}. Assume $\alpha, \delta, \theta$ to be positive constants whose values will be determined later in this section. As mentioned before, we use notation $x$ instead of $x_1$. 
Let 
\begin{equation}\label{dataw0}
\tilde{w}^{(\eta)}_1(x,x_2)=w_1^{(\eta)}(x,x_2,0)=-\theta |\ln (x)|^\alpha \mathcal{X}(\frac{x}{\eta})\psi\Big(\frac{|\ln(x)|^\delta x_2}{\sqrt{x}}\Big),
\end{equation}
where $0<\delta<1$ and
$$\mathcal{X}(x)=\left\{\begin{split}
&1,\quad x\in[\frac65,\frac95],\\
&0,\quad x\in(-\infty, 1]\cup [2,+\infty),
\end{split}\right.
\qquad \psi(x)=\left\{\begin{split}
&1,\quad |x|\leq \frac14,\\
&0,\quad |x|\geq\frac12.
\end{split}\right.
$$ 
We further assign
\begin{equation} \label{data2}
\tilde{w}^{(\eta)}_k(x,x_2)=w_k^{(\eta)}(x,x_2,0)=-\theta^2 \mathcal{X}(\frac{x}{\eta})\psi\Big(\frac{|\ln(x)|^\delta x_2}{\sqrt{x}}\Big)\quad\text{for}\quad k=2,\cdots,5.
\end{equation}
With the above initial data, we prove the following lemma.

\begin{lem} \label{data lem}
Let $\tilde{w}^{(\eta)}_1$ be defined as in \eqref{dataw0}. Then, for $s\leq \frac34$, there holds $$\|\tilde{w}^{(\eta)}_1\|_{\dot{H}^{s}(\mathbb{R}^2)}\lesssim \theta \eta^{\frac34-s} |\ln\eta|^{\alpha-\frac{\delta}{2}}.$$
\end{lem}
\begin{proof}
Employing Gagliardo-Nirenberg inequality, for $0<s<1$, the $\dot{H}^{s}(\mathbb{R}^2)$ norm of $\tilde{w}_1$ is controlled by
\begin{equation}\label{snorm}
\|\tilde{w}^{(\eta)}_1\|_{\dot{H}^s(\mathbb{R}^2)}\lesssim \|\tilde{w}^{(\eta)}_1\|_{L^2(\mathbb{R}^2)}^{1-s}\|\tilde{w}^{(\eta)}_1\|_{H^1(\mathbb{R}^2)}^s.
\end{equation}
We first obtain the $L^2$ estimate of $\tilde{w}^{(\eta)}_1$ as follows:
\begin{equation}\label{L2}
\begin{split}
\|\tilde{w}^{(\eta)}_1\|_{L^2(\mathbb{R}^2)}^2\lesssim& \theta^2\int_{x\in[\eta,2\eta]}\int_{x_2\leq\sqrt{x}|\ln x|^{-\delta}}|\ln x|^{2\alpha}dx_2dx\\
\lesssim&\theta^2\int_{x\in[\eta,2\eta]}\sqrt{x}|\ln x|^{2\alpha-\delta}dx\\
\lesssim&\theta^2\eta^{\frac32}|\ln  \eta|^{2\alpha-\delta}.
\end{split}
\end{equation}
Then, by using the construction of  $\tilde{w}^{(\eta)}_1$ in \eqref{dataw0}, we compute the first derivatives:
\begin{equation} \label{express d1 w1}
\begin{split}
\partial_x \tilde{w}^{(\eta)}_1=&-\theta\alpha\frac{ |\ln x|^{\alpha-1}}{x}\mathcal{X}(\frac{x}{\eta})\psi\Big(\frac{|\ln(x)|^\delta x_2}{\sqrt{x}}\Big)-\theta \frac{|\ln x|^{\alpha}}{\eta}\mathcal{X}'(\frac{x}{\eta})\psi\Big(\frac{|\ln(x)|^\delta x_2}{\sqrt{x}}\Big)\\
&-\theta|\ln x|^{\alpha}\mathcal{X}(\frac{x}{\eta})\psi'\Big(\frac{|\ln(x)|^\delta x_2}{\sqrt{x}}\Big)\Big(\delta|\ln x|^{\delta-1}x^{-\frac32}x_2-\frac12|\ln x|^\delta x^{-\frac32}x_2\Big),
\end{split}
\end{equation}
and 
\begin{equation} \label{d2 w1}
\partial_{x_2}\tilde{w}^{(\eta)}_1=-\theta \mathcal{X}(\frac{x}{\eta})\frac{|\ln x|^{\alpha+\delta}}{\sqrt{x}}\psi'\Big(\frac{|\ln(x)|^\delta x_2}{\sqrt{x}}\Big).
\end{equation}
We are now ready to calculate the $\dot{H}^1$ norm of  $\tilde{w}^{(\eta)}_1$, i.e. 
\begin{equation}
\|\tilde{w}_1^{(\eta)}\|_{\dot{H}^1(\mathbb{R}^2)}^2=\|\partial_x \tilde{w}_1^{(\eta)}\|_{L^2(\mathbb{R}^2)}^2+\|\partial_{x_2}\tilde{w}_1^{(\eta)}\|_{L^2(\mathbb{R}^2)}^2.
\end{equation}
By using \eqref{express d1 w1}-\eqref{d2 w1}, we obtain
\begin{equation} \label{d1 norm est}
	\begin{split}
		\|\partial_{x}\tilde{w}_1^{(\eta)}\|_{L^2(\mathbb{R}^2)}^2\lesssim&\theta^2\int_{x\in[\eta,2\eta]}\int_{x_2\leq\sqrt{x}|\ln x|^{-\delta}}|\ln x|^{2\alpha-2}x^{-2}dx_2dx\\
		&+\theta^2\int_{x\in[\eta,2\eta]}\int_{x_2\leq\sqrt{x}|\ln x|^{-\delta}}|\ln x|^{2\alpha}\eta^{-2}dx_2dx\\
        &+\theta^2\int_{x\in[\eta,2\eta]}\int_{x_2\leq\sqrt{x}|\ln x|^{-\delta}}|\ln x|^{2\alpha}\big(|\ln x|^{\delta-1}x^{-\frac32}x_2\big)^2dx_2dx\\
        &+\theta^2\int_{x\in[\eta,2\eta]}\int_{x_2\leq\sqrt{x}|\ln x|^{-\delta}}|\ln x|^{2\alpha}\big(|\ln x|^{\delta}x^{-\frac32}x_2\big)^2dx_2dx\\
		\lesssim&\theta^2\eta^{-\frac12}|\ln \eta|^{2\alpha-\delta},
	\end{split}
\end{equation}
and
\begin{equation} \label{x2 norm}
\begin{split}
\|\partial_{x_2}\tilde{w}_1^{(\eta)}\|_{L^2(\mathbb{R}^2)}^2\lesssim&\theta^2\int_{x\in[\eta,2\eta]}\int_{x_2\leq\sqrt{x}|\ln x|^{-\delta}}|\ln x|^{2\alpha+2\delta}x^{-1}dx_2dx\\
\lesssim&\theta^2\int_{x\in[\eta,2\eta]}x^{-\frac12}|\ln x|^{2\alpha+\delta}dx\\
\lesssim&\theta^2\eta^{\frac12}|\ln \eta|^{2\alpha+\delta}.
\end{split}
\end{equation}
Therefore, the $\dot{H}^1$ norm of  $\tilde{w}_1^{(\eta)}$ satisfies 
\begin{equation}\label{H1}
\|\tilde{w}_1^{(\eta)}\|_{\dot{H}^1(\mathbb{R}^2)}\lesssim \theta\eta^{-\frac14}|\ln \eta|^{\alpha-\frac{\delta}{2}}.
\end{equation}
Inserting \eqref{L2} and \eqref{H1} into the Gagliardo-Nirenberg inequality \eqref{snorm}, we have
\begin{equation}\label{Hs}
\|\tilde{w}_1^{(\eta)}\|_{\dot{H}^s(\mathbb{R}^2)}\lesssim \theta \eta^{\frac34-s}|\ln \eta|^{\alpha-\frac{\delta}{2}}.
\end{equation}
 This completes the proof of this lemma.
\end{proof}

\begin{rmk}
    By choosing a sequence $\{\eta_m\}$ with sufficiently rapid decay speed, for instance, $\eta_m \sim e^{-\frac{1}{m^M}}$ with $M\geq 2$ being sufficiently large, one can see that the initial data have disjoint supports, and their sum $\tilde{w}_1=\sum\limits_m\tilde{w}_1^{(\eta_m)}$ converges in the corresponding Sobolev space. This $\tilde{w}_1$ then provides a counterexample to local existence of low regularity solutions. Noting that $T_{\eta_m}^* \to 0$, there exists no solution with initial data $\tilde{w}_1$ for any $T>0$.
\end{rmk}

\begin{rmk}
	In 3D, we can construct a  similar profile
	\begin{equation}\label{dataw0 3D}
		w_1^{(\eta)}(x,x_2,x_3,0)=-\theta |\ln (x)|^\alpha \mathcal{X}(\frac{x}{\eta})\psi\Big(\frac{|\ln(x)|^\delta x_2}{\sqrt{x}}\Big)\psi\Big(\frac{|\ln(x)|^\delta x_3}{\sqrt{x}}\Big).	\end{equation}
We verify in \cite{an2} that this profile leads to $H^2$ ill-posedness in 3D.
\end{rmk}
\begin{rmk}
Note that we can consider a more general type of initial data
\begin{equation*}
\tilde{w}^{(\eta)}_1(x,x_2)=-\theta |\ln (x)|^\alpha \mathcal{X}(\frac{x}{\eta})\psi\Big(\frac{|\ln(x)|^\delta x_2^{k}}{x^m}\Big).
\end{equation*}
The ansatz is related to the threshold regularity $H^\frac34$. In particular, in the same manner, for $\tilde{w}^{(\eta)}_1$ we have
\begin{equation*}
\|\tilde{w}_1^{(\eta)}\|^2_{\dot{H}^{s}(\mathbb{R}^2)}
\lesssim\theta^2\underline{\eta^{1+\frac{m}{k}-2s}}|\ln\eta|^{2\alpha-\frac{\delta}{k}}.
\end{equation*}
Here, $s=\frac12+\frac{m}{2k}$ is the threshold regularity. In \eqref{dataw0}, we take $\frac{m}{k}=\frac12$. The $x_2$-scale of this new datum is $\frac{\eta^\frac{m}{k}}{|\ln \eta|^\frac{\delta}{k}}$. As discussed in Section \ref{ill}, to ensure that the planar symmetric region covers the initial domain of dependence, we must require that
\begin{equation*}
\frac{\sqrt{\eta}}{\theta^{\frac12}|\ln \eta|^\frac{\alpha}{2}}\lesssim\frac{\eta^\frac{m}{k}}{|\ln \eta|^\frac{\delta}{k}},
\end{equation*}
which necessarily indicates that $\frac{m}{k}\leq\frac12$ for $\theta\sim |\ln\eta|^{2\delta-\alpha}$. Here $\frac{m}{k}=\frac12$ optimizes the threshold regularity, which then yields the criticality of $s=\frac34$. See Figure \ref{threshold fig} below as illustration.
\begin{figure}[H]
\centering
\includegraphics[width=0.7\linewidth]{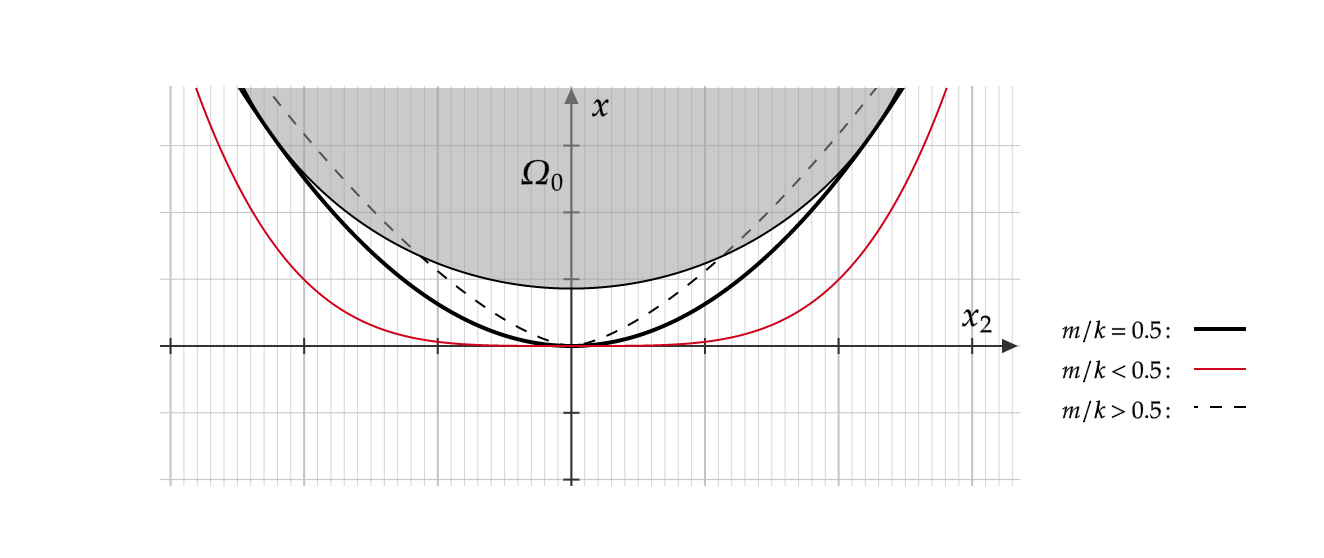}
\captionof{figure}{}
\label{threshold fig}
\end{figure}
\end{rmk}

From \eqref{Hs}, one can observe that $s=\frac34$ is the threshold regularity. To establish ill-posedness in $H^{s+1}$, we carefully analyze the relationship of the parameters $\theta,\alpha, \delta$. As will be discussed in Section \ref{ill}, we choose  $\theta\sim|\ln\eta|^{2\delta-\alpha}\ll 1$, and impose the following conditions for the parameters 
\begin{equation}
	\alpha>2\delta.
\end{equation}
With this construction, for $s<\frac34$, we have that $\|w_1^{(\eta)}\|_{\dot{H}^{s}(\mathbb{R}^2)}\to 0$ as $\eta\to 0$. Based on this, we now go back to the original system \eqref{eq} by reversing the process of wave decomposition. By the formula of decomposition of waves \eqref{phix}, we have
\begin{equation} \label{combo2}
\partial_x\Phi(x,x_2,0)=\sum_{i=1}^5 {w}_i^{(\eta)}(x,x_2,0)r_i\big(\Phi(x,x_2,0)\big).
\end{equation}
Noting that the right eigenvectors $r_i$ chosen in Section \ref{pre} are Lipschitz continuous in $\Phi$, by the ODE argument, one can solve \eqref{combo2} for the initial data $\Phi^{(\eta)}(x,x_2,0)$ of the 2D system \eqref{eq}. Moreover, recall from Section \ref{sbd} that $r_i(\Phi)\in L^\infty( B_{2\delta}^5(0))$. Together with the fact that $w^{(\eta)}_1(x,x_2,0)\in H^{s}(\mathbb{R}^2)$, it follows that $\partial_x\Phi\in H^{s}(\mathbb{R}^2).$ Furthermore, when $j \neq 1$, we have $\|\tilde{w}_j^{(\eta)}\|_{{H}^{s}(\mathbb{R}^2)}\leq \|\tilde{w}_1^{(\eta)}\|_{{H}^{s}(\mathbb{R}^2)}$. Together with \eqref{Hs}, we know that when $s<\frac34$ this implies the vanishing of $\|\partial_x\Phi_0^{(\eta)}\|_{\dot{H}^{s}(\mathbb{R}^2)}$ as $\eta \to 0$ and whence the $H^{\frac74-\epsilon}$  ill-posedness\footnote{The vanishing of $\|\partial_{x_2}\Phi_0^{(\eta)}\|_{\dot{H}^{s}(\mathbb{R}^2)}$ can be obtained in a similar fashion as in \eqref{x2 norm}.} with arbitrary $0<\epsilon\ll 1$ for the Cauchy problems of the 2D ideal compressible MHD system.

\begin{rmk} \label{l2bump}
Here, we also need to guarantee the $L^2$-integrability of $\Phi_0$. According to \eqref{combo2}, this is equivalent to show the $L^2$-integrability of $\int_{0}^{x}\tilde{w}_1^{(\eta)}(s)ds$. This could be achieved by multiplying the data by a cut-off function independent of $\eta$, for instance, the smooth cut-off function $\psi(x)$ defined in \eqref{dataw0} which is compactly supported within $\{|x|\leq\frac12\}$\footnote{The choice of cut-off function does not affect the shock formation argument, since the (planar) shock forms only within a small interval $x\in[\eta,2\eta]$, where the cut-off function satisfies $\psi(x)\equiv 1$.}. In particular, we have that $\psi(x)\int_{0}^{x}\tilde{w}_1^{(\eta)}(s)ds \in L^2(\mathbb{R}^2)$. Noting that
\begin{equation}\label{c0func}
\left|\psi(x)\int_{0}^{x}\tilde{w}_1^{(\eta)}(s)ds\right| \leq \left|\psi(x)\int_{\eta}^{2\eta}\tilde{w}_1^{(\eta)}(s)ds\right|\lesssim \eta |\ln \eta|^\alpha,
\end{equation}
one could verify that the $L^2$-norm of $\psi(x)\int_{0}^{x}\tilde{w}_1^{(\eta)}(s)ds$ also vanishes as $\eta\to 0$.

Furthermore, we take first derivative of this function and obtain that
\begin{equation}\label{c1func}
\frac{d}{dx}\left(\psi(x)\int_{0}^{x}\tilde{w}_1^{(\eta)}(s)ds\right)= \psi'(x)\int_{0}^{x}\tilde{w}_1^{(\eta)}(s)ds+\psi(x)\tilde{w}_1^{(\eta)}(x).
\end{equation}
The second term on the RHS of \eqref{c1func} has already been estimated in this section. For the first term on the RHS of \eqref{c1func}, by a similar argument as for \eqref{c0func}, we have that
$$
\|\psi'(x)\int_{0}^{x}\tilde{w}_1^{(\eta)}(s)ds_1^{(\eta)}\|^2_{\dot{H}^{s}(\mathbb{R}^2)}
 \lesssim \eta^2\|\tilde{w}_1^{(\eta)}\|^2_{\dot{H}^{s}(\mathbb{R}^2)},
$$
which also vanishes as $\eta\to 0$. In the following sections, without loss of generality, we use the initial data defined in \eqref{dataw0} and \eqref{data2}. The additional term appearing in this remark is negligible when $\eta$ is sufficiently small.
\end{rmk}

 In the critical space $H^\frac74$, i.e., when $s=\frac34$, the situation becomes more subtle. As we will see in Section \ref{ill}, to accommodate the planar symmetric shock formation argument, in $x_2$ direction, the initial domain of dependence must be at least of width $\frac{\eta^\frac12}{\theta^\frac12|\ln \eta|^\frac{\alpha}{2}}$. On the other hand, due to our construction of initial data in \eqref{dataw0}, the data exhibit plane symmetry within a domain satisfying
$$
|x_2|\sim \frac{\sqrt{x}}{|\ln x|^\delta}.
$$
Since our data are supported in $x$-direction within the interval $[\eta,2\eta]$, the above formula implies that
$$
|x_2|\sim \frac{\sqrt{\eta}}{|\ln \eta|^\delta}.
$$
Comparing this with the size of the initial domain of dependence, we need $\theta\gtrsim \left|\ln \eta\right|^{-\alpha+2\delta}$ to ensure that the initial domain of dependence lies in the planar symmetric portion of the initial data. In particular, as mentioned before, we choose  $\theta\sim|\ln\eta|^{2\delta-\alpha}\ll 1$ with
$2\delta-\alpha<0$.
Then, the lifespan verifies the following condition:
\begin{equation}
T_\eta^*\lesssim\frac{1}{|\ln\eta|^{2\delta-\alpha+\alpha}}\sim\frac{1}{|\ln\eta|^{2\delta}} \quad \text{with}\quad \delta>0,
\end{equation}
which tends to $0$ as $\eta\to 0$. The instantaneous shock formation and $H^1$ norm inflation stated in Theorem \ref{2D} therefore also apply to this new family of initial data. In this case, the initial data satisfy
\begin{equation} \label{growing norm}
\|w_1^{(\eta)}\|_{\dot{H}^{\frac34}(\mathbb{R}^2)}\lesssim |\ln\eta|^{2\delta-\alpha+\alpha-\frac{\delta}{2}}=|\ln\eta|^{\frac{3\delta}{2}}.
\end{equation}
The exponent of $\left|\ln \eta\right|$ is positive, causing this estimate to blow up as $\eta \to 0$. However, by comparing \eqref{d1 norm est} with \eqref{x2 norm}, we observe that the regularity loss concentrates entirely along the $x_1$ direction. To address this issue, we will introduce the anisotropic Sobolev space and discuss the $H^\frac74$ ill-posedness in the anisotropic sense. Before proceeding with the anisotropic ill-posedness, we present two remarks regarding the isotropic case.
\begin{rmk} \label{failiure reason}
 The unexpected initial behaviour \eqref{growing norm} results from a technical reason, namely the planar symmetric condition in shock formation argument. Recall that our choice of parameters $\theta, \alpha, \delta$ is based on three essential conditions:
 \begin{enumerate}
     \item Instantaneous shock formation, i.e.,  $$T_\eta^*\sim(\theta|\ln \eta|^\alpha)^{-1} \to0.$$
     \item Vanishing initial data, i.e.,  $$\|\tilde{w}_1^{(\eta)}\|_{\dot{H}^{\frac34}}\lesssim \theta |\ln \eta|^{\alpha-\frac{\delta}{2}} \to 0.$$
     \item The planar symmetric region is large enough, i.e., $$\theta\gtrsim \left|\ln \eta\right|^{-\alpha+2\delta}.$$
 \end{enumerate}
 The preceding analysis shows that $(1)(2)(3)$ cannot be satisfied simultaneously. However, among these three conditions, condition $(3)$ is purely technical and stems from the symmetry requirement in the current shock formation framework. If the shock formation argument can be extended to the general situation with no symmetry, then condition $(1)$ would still yield the desired blow-up time. In that case, one could find admissible parameters such that $(1)(2)$ are both achieved, for example, $\theta=|\ln \eta|^{-\alpha+\delta/3}$. 
 
 Therefore, the obstruction encountered here is not intrinsic to the mechanism itself, but rather to the present technical framework. In principle, this issue could be resolved by establishing a more robust shock formation result with no symmetry, and this will yield the isotropic $H^\frac74$ ill-posedness.   
\end{rmk}
\begin{rmk} \label{weak ill-posedness}
In the isentropic space $H^\frac74$, we still have the following version of ill-posedness. Specifically, we set $\theta=\left|\ln \eta\right|^{-\alpha+2\delta}$ with $\delta<0$. Under this choice, conditions $(2)$ and $(3)$ in Remark \ref{failiure reason} are satisfied, at the cost of the increase in blow-up time $T_\eta^*$. This indicates the existence of a family of small initial data with their sizes shrinking to $0$, whereas each of the initial data lead to shock formation in a future time.  
\end{rmk}

Now, at the threshold regularity $s=\frac74$, we investigate the ill-posedness in the anisotropic space. First, we define the anisotropic Sobolev space\footnote{Here, we consider the two-dimensional case. For further discussion of anisotropic Sobolev spaces in a more general framework, we refer interested readers to Bahouri-Chemin-Danchin \cite{ba-ch-da book}, Besov \cite{besov} and Slobodecki\u{\i} \cite{aniso paper}.} $H^{s_1,s_2}(\mathbb{R}^2)$.
\begin{definition} \label{aniso space}
Let $s_1$ and $s_2$ be real numbers. The anisotropic Sobolev space $H^{s_1,s_2}(\mathbb{R}^2)$ consists of all tempered distributions $f$ such that 
\begin{equation*}
\|f\|_{H^{s_1,s_2}}=\Big(\int_{\mathbb{R}^2}(1+|\xi_1|^2)^{s_1}(1+|\xi_2|^2)^{s_2}|\hat{f}(\xi)|^2d\xi\Big)^{1/2}<+\infty.
\end{equation*}
Similarly, the homogeneous anisotropic Sobolev space $\dot{H}^{s_1,s_2}(\mathbb{R}^2)$ is defined as the set of tempered distributions $f$ satisfying
\begin{equation*}
\|f\|_{\dot{H}^{s_1,s_2}}=\Big(\int_{\mathbb{R}^2}|\xi_1|^{2s_1}|\xi_2|^{2s_2}|\hat{f}(\xi)|^2d\xi\Big)^{1/2}<+\infty.
\end{equation*}
\end{definition}
In the following, for notational simplicity, we denote $\partial_1^{s_1}\partial_2^{s_2} f$ as the inverse Fourier transform of $\xi_1^{s_1}\xi_2^{s_2}\hat{f}$. For each fixed $i=1,2$,  the interpolation inequality yields
\begin{equation}\label{3.25}
    \|\partial_{x_i}^{s_i}f\|_{L^2}\lesssim \|f\|_{L^2}^{1-s_i}\|\partial_{x_i}f\|_{L^2}^{s_i}.
\end{equation}
For the family of initial data 
\begin{equation}
		w_1^{(\eta)}(x_1,x_2)=-\theta |\ln (x_1)|^\alpha \mathcal{X}(\frac{x_1}{\eta})\psi\Big(\frac{|\ln(x_1)|^\delta x_2}{\sqrt{x_1}}\Big),	\end{equation}
we now prove that these data vanish in $H^{\frac34-\epsilon,\epsilon}(\mathbb{R}^2)$ as $\eta \to 0$ for any $0<\epsilon\leq\frac34$. Employing the interpolation inequality twice, we derive
\begin{equation}\label{3.27}
\begin{aligned}
    \|\partial_{x_1}^{\frac34-\epsilon}\partial_{x_2}^\epsilon w_1^{(\eta)}\|_{L^2}&\lesssim\|\partial_{x_2}^\epsilon w_1^{(\eta)}\|_{L^2}^{\frac14+\epsilon}\|\partial_{x_1}\partial_{x_2}^\epsilon w_1^{(\eta)}\|_{L^2}^{\frac34-\epsilon}\\
    &\lesssim\Big(\|w_1^{(\eta)}\|_{L^2}^{1-\epsilon}\|\partial_{x_2}w_1^{(\eta)}\|_{L^2}^{\epsilon}\Big)^{\frac14+\epsilon}\Big(\|\partial_{x_1}w_1^{(\eta)}\|_{L^2}^{1-\epsilon}\|\partial_{x_2}\partial_{x_1}w_1^{(\eta)}\|_{L^2}^{\epsilon}\Big)^{\frac34-\epsilon}.
    \end{aligned}
\end{equation}
In order to control the RHS of \eqref{3.27}, we compute the higher order derivative:
\begin{equation} 
\begin{split}
\partial_{x_2}\partial_{x_1} w^{(\eta)}_1=&-\theta\alpha\frac{ |\ln x_1|^{\alpha+\delta-1}}{x_1^{3/2}}\mathcal{X}(\frac{x_1}{\eta})\psi'\Big(\frac{|\ln x_1|^\delta x_2}{\sqrt{x_1}}\Big)-\theta \frac{|\ln x_1|^{\alpha+\delta}}{\eta\sqrt{x_1}}\mathcal{X}'(\frac{x_1}{\eta})\psi'\Big(\frac{|\ln x_1|^\delta x_2}{\sqrt{x_1}}\Big)\\
&-\theta|\ln x_1|^{\alpha}\mathcal{X}(\frac{x_1}{\eta})\psi''\Big(\frac{|\ln x_1|^\delta x_2}{\sqrt{x_1}}\Big)\Big(\delta|\ln x_1|^{2\delta-1}x_1^{-2}x_2-\frac12|\ln x_1|^{2\delta} x_1^{-2}x_2\Big)\\
&-\theta|\ln x_1|^{\alpha}\mathcal{X}(\frac{x_1}{\eta})\psi'\Big(\frac{|\ln x_1|^\delta x_2}{\sqrt{x_1}}\Big)\Big(\delta|\ln x_1|^{\delta-1}x_1^{-\frac32}-\frac12|\ln x_1|^{\delta} x_1^{-\frac32}\Big).
\end{split}
\end{equation}
We bound its $L^2$-norm as follows:
\begin{equation}
    \|\partial_{x_2}\partial_{x_1}w_1^{(\eta)}\|_{L^2}\lesssim\theta\eta^{-\frac34}|\ln\eta|^{\alpha+\frac\delta2}.
\end{equation}
For other lower order terms in \eqref{3.27}, recall from the proof of Lemma \ref{data lem} that
\begin{align}   &\|w_1^{(\eta)}\|_{L^2}\lesssim\theta\eta^{\frac34}|\ln\eta|^{\alpha-\frac\delta2},\\
&\|\partial_{x_1}w_1^{(\eta)}\|_{L^2}\lesssim\theta\eta^{-\frac14}|\ln\eta|^{\alpha-\frac\delta2},\\
&\|\partial_{x_2}w_1^{(\eta)}\|_{L^2}\lesssim\theta\eta^{\frac14}|\ln\eta|^{\alpha+\frac\delta2}.
\end{align}
Plugging in all the above estimates in \eqref{3.27}, by the definition of the anisotropic Sobolev space, we then obtain the following estimate
\begin{equation}
\begin{aligned}   \|w_1^{(\eta)}\|_{H^{\frac34-\epsilon,\epsilon}}\lesssim& \|w_1^{(\eta)}\|_{L^2}+\|\partial_{x_1}^{\frac34-\epsilon}\partial_{x_2}^\epsilon w_1^{(\eta)}\|_{L^2}\\
   \lesssim&\theta\eta^{\frac34}|\ln\eta|^{\alpha-\frac\delta2}+\theta\eta^{\frac12\epsilon}|\ln\eta|^{\alpha-\frac\delta2+\epsilon\delta},
\end{aligned}
\end{equation}
which goes to $0$ as $\eta\to 0$. Consequently, this construction yields the anisotropic $H^\frac74$ ill-posedness.

In addition, we need to solve $H_1$ to preserve the constraint $\nabla\cdot H=0$, especially in the non-symmetric region. We define the initial data of $H_2$ as follows\footnote{In Theorem \ref{2D}, we replace the initial data of $\Phi_4^{(\eta)}$ with this $H_2^{(\eta)}(x,x_2,0)$.},
    \begin{equation*}
        H_2^{(\eta)}(x,x_2,0)=
        \begin{cases}
        \Phi_4^{(\eta)}(x,x_2,0), \quad x\geq0,\\
        -\Phi_4^{(\eta)}(-x,x_2,0), \quad x<0.
        \end{cases}
    \end{equation*}
Note that $\Phi_4^{(\eta)}(x,x_2,0)$ is indeed supported within $\{\eta\leq x\leq 2\eta\}$. Therefore, by finite speed of propagation, the $\{x<0\}$-part of $H_2^{(\eta)}(x,x_2,0)$ will not interfere with the previous shock formation and ill-posedness argument carried out in $\{x\geq 0\}$. Assume that the constant parameter $H_1$ takes the value $\kappa$ in the earlier planar symmetric shock formation argument. We first assign $H_1(x,x_2)\equiv \kappa$ for $x\geq 2\eta$. Then, for each fixed $x_2$, we solve the transport equation for $H_1^{(\eta)}$ backward in $x$:
    $$\partial_{x_1} H_1^{(\eta)}(x,x_2,0)=-\partial_{x_2}H_2^{(\eta)}(x,x_2,0),$$
with the initial condition $$H_1^{(\eta)}(2\eta,x_2,0)=\kappa.$$
Hence, the initial datum of $H_{1}^{(\eta)}-\kappa$ is compactly supported within $[-2\eta,2\eta]$. Moreover, by Remark \ref{l2bump}, $\|H_{1}^{(\eta)}(\vec{x},0)-\kappa\|_{L^\infty(\mathbb{R}^2)}$ is bounded by $\eta |\ln \eta|^\alpha$, which implies $$\|H_{1}^{(\eta)}(\vec{x},0)-\kappa\|_{L^2(\mathbb{R}^2)}\to 0,\quad \text{as}\quad \eta \to 0.$$ 
In addition, by Lemma \ref{data lem}, we can also establish the vanishing of $\|\partial H_{1}^{(\eta)}\|_{H^s(\mathbb{R}^2)}$ at $t=0$. Therefore, the initial data of $H_1^{(\eta)}-\kappa$ here constitute the profile we need.

\section{Shock formation}\label{aprior}
In this section, we prove the (first) shock forms at a finite time $T_\eta^*$ within $\mathcal{R}_1$. To demonstrate this, we carry out a priori estimates in $L^\infty$ norm.
\subsection{$L^\infty$ estimates} \label{Linfest}
Consider the decomposed diagonal system \eqref{eqrho}-\eqref{eqv} under planar symmetry. For the initial data designed in Section \ref{data}, we have
\begin{equation*}
W_0^{(\eta)}:=\max_i\sup_{z}|w_i^{(\eta)}(z,0)|=-w_1^{(\eta)}(z_0,0)>0
\end{equation*}
for some $z_0\in(\eta, 2\eta)$. To study the Cauchy problem for system \eqref{eqrho}-\eqref{eqv}, we estimate the following quantities:
\begin{align*}
  S_\eta(t):=&\max_{i}\sup_{(z'_i,s'_i)\atop z'_i\in[\eta,2\eta],\ 0\leq s'_i\leq t}\rho_i^{(\eta)}(z'_i,s'_i),\\
   J_\eta(t):=&\max_{i}\sup_{(z'_i,s'_i)\atop z'_i\in[\eta,2\eta]\ 0\leq s'_i\leq t}|v_i^{(\eta)}(z'_i,s'_i)|,\\
  V_\eta(t):=&\max_i\sup_{(x',t')\notin\mathcal{R}_i,\atop 0\leq t'\leq t}|w_i^{(\eta)}(x',t')|,\\
  \bar{U}_\eta(t):=&\sup_{(x',t')\atop 0\leq t'\leq t}|\Phi^{(\eta)}(x',t')|.
\end{align*}
Let $\Phi^{(\eta)}\in C^2(\mathbb{R}\times[0,T],B_{2\delta}^5(0))$ be a solution to \eqref{MHD} for $t\in[0,T]$ with $T>0$. Our aim is to show that, for  $t\in[0,T_{\eta}^*)$ it holds
\begin{align}\label{bound}
  S_\eta(t)=O(1),\ J_\eta(t)=O(W_0^{(\eta)}),\
 V_\eta(t)=O\Big(\eta [W_0^{(\eta)}]^2\Big),\ \bar{U}_\eta(t)&=O(\eta W_0^{(\eta)}),
\end{align}
and
\begin{equation}
T_{\eta}^*\leq O\Big(\frac{1}{W_0^{(\eta)}}\Big).
\end{equation}
Here, $O(\cdot)$ again, denotes the Landau notation on the whole time interval $[0,T_\eta^*)$. For simplicity, we suppress the subscript $\eta$ in the above quantities, with $\eta$ viewed as a fixed small parameter in this section.

The main strategy is to derive the following a priori estimates for $t<T_\eta^*$
\begin{align}
  S(t)&=O(1+tJ+tVS),\label{bds} \\
  J(t)&=O(W_0^{(\eta)}+t VJ+tV^2 S),\label{bdj}\\
 V(t)&=O\Big(\eta [W_0^{(\eta)}]^2+tV^2+\eta VJ\Big),\label{bdv}\\
 \bar{U}(t)&=O(\eta J+\eta V+\eta tV).\label{bdu}
\end{align}
Once these a priori estimates are obtained, we then set up a bootstrap argument to deduce \eqref{bound}. In particular, supposing that
\begin{equation} \label{BA}
 tV\leq \theta^{\frac12}, \qquad J\leq \theta^{\frac12},
\end{equation}
via \eqref{bds}-\eqref{bdv}, we have that
\begin{equation*}
\begin{split}
S(t)=O(1+tJ+\theta^{\frac12}S)&\qquad\Rightarrow\qquad S(t)=O(1+tJ),\\
 V(t)=O\Big(\eta [W_0^{(\eta)}]^2+\theta^{\frac12}V+\eta  \theta^{\frac12}V\Big)&\qquad\Rightarrow\qquad V(t)=O\Big(\eta [W_0^{(\eta)}]^2\Big),
\end{split}
\end{equation*}
and
\begin{equation*}
\begin{split}
J(t)=O(W_0^{(\eta)}+\theta^{\frac12}J+\theta^{\frac12}VS)\qquad&\Rightarrow\qquad J(t)=O\big(W_0^{(\eta)}+\theta^{\frac12}V(1+tJ)\big)\\
&\Rightarrow\qquad J(t)=O\big(W_0^{(\eta)}+\theta^{\frac12}V\big)=O\big(W_0^{(\eta)}\big)
\end{split}
\end{equation*}
Then for $t< O(\frac{1}{W_0^{(\eta)}})$, we get $S(t)=O(1)$. These estimates improve the bootstrap assumptions in \eqref{BA}. Moreover, with \eqref{bdu} we also obtain $\bar{U}(t)=O(\eta W_0^{(\eta)})$.
\\

Now we derive the a priori estimates stated in \eqref{bds}-\eqref{bdu}. Recall that for any fixed $x_2$, the initial data are supported in $[\eta,2\eta]$. And from \eqref{t0}, we also know that the characteristic strips are well separated after time $$t_0^{(\eta)}:=\frac{\eta}\sigma.$$
\begin{figure}[H]
\centering
\begin{tikzpicture}[fill opacity=0.5, draw opacity=1, text opacity=1]
\node [below]at(3.5,0){$2\eta$};
\node [below]at(2.2,0){$\eta$};

\filldraw[white, fill=gray!40] (3.5,0)..controls (2,1) and (1,2)..(-0.8,3.5)--(0.4,3.5)..controls (1,2.8) and (1.9,2)..(2.5,0);
\filldraw[white, fill=gray!40](3.5,0)..controls (3.2,1) and (2,2.6)..(1.3,3.5)--(3.5,3.5)..controls (2.2,1.5) and (2.6,1)..(2.5,0);
\filldraw[white, fill=gray!40](3.5,0)..controls (3.4,1) and (3,1.5)..(4.5,3.5)--(5.5,3.5)..controls (3.2,1.5) and (2.8,1)..(2.5,0);
\filldraw[white, fill=gray!40](3.5,0)..controls (3.8,1) and (4,1.5)..(6.6,3.5)--(8.5,3.5)..controls (3.9,1.5) and (3.5,1)..(2.5,0);

\filldraw[white, fill=gray!80] (2.5,0)..controls (3.5,1) and (3.9,1.5)..(8.5,3.5)--(10,3.5)..controls (5.5,1.5) and (4,0.5)..(3.5,0)--(2.5,0);
\filldraw[white,fill=gray!80](2.5,0)..controls (2.8,1) and (3.2,1.5)..(5.5,3.5)--(6.6,3.5)..controls (4,1.5) and (3.8,1)..(3.5,0);
\filldraw[white,fill=gray!80](2.5,0)..controls (2.6,1) and (2.2,1.5)..(3.5,3.5)--(4.5,3.5)..controls (3,1.5) and (3.4,1)..(3.5,0);
\filldraw[white,fill=gray!80](2.5,0)..controls (1.9,2) and (1,2.8)..(0.4,3.5)--(1.3,3.5)..controls (2,2.6) and (3.2,1)..(3.5,0);
\filldraw[white,fill=gray!80](2.5,0)..controls (1,1) and (0.5,1.8)..(-1.8,3.5)--(-0.8,3.5)..controls (1,2) and (2,1)..(3.5,0);
\draw[->](-1.8,0)--(10,0)node[left,below]{$t=0$};
\draw[dashed](-1.8,1.7)--(10,1.7)node[right,below]{$t=t_0^{(\eta)}$};

\draw (3.5,0)..controls (4,0.5) and (5.5,1.5)..(10,3.5);

\draw (2.5,0)..controls (3.5,1) and (3.9,1.5)..(8.5,3.5);
\node [below] at(7.6,3){$\mathcal{R}_1$};

\draw [color=black](3.5,0)..controls (3.8,1) and (4,1.5)..(6.6,3.5);

\draw [color=black](2.5,0)..controls (2.8,1) and (3.2,1.5)..(5.5,3.5);
\node [below] at(5.2,3){$\mathcal{R}_2$};

\draw [color=black](3.5,0)..controls (3.4,1) and (3,1.5)..(4.5,3.5);

\draw [color=black](2.5,0)..controls (2.6,1) and (2.2,1.5)..(3.5,3.5);
\node [below] at(3.5,3){$\mathcal{R}_3$};

\draw [color=black](3.5,0)..controls (3.2,1) and (2,2.6)..(1.3,3.5);
\node [below] at(1.6,3){$\mathcal{R}_4$};

\draw [color=black] (2.5,0)..controls (1.9,2) and (1,2.8)..(0.4,3.5);

\draw [color=black](3.5,0)..controls (2,1) and (1,2)..(-0.8,3.5);
\node [below] at(-0.2,3){$\mathcal{R}_{5}$};

\draw [color=black] (2.5,0)..controls (1,1) and (0.5,1.8)..(-1.8,3.5);
\end{tikzpicture}
\caption{\small\bf Separation of five characteristic strips.} \label{fig1}
\end{figure}
According to Section \ref{sbd}, all the coefficients in equations \eqref{eqrho}-\eqref{eqv} are proved to be bounded. Here we denote this uniform upper bound to be $\Gamma$. The a priori estimates are deduced in two regions: the non-separating region $[0,t_0^{(\eta)}]$ and the separating region $[t_0^{(\eta)},T]$ as in Figure \ref{fig1}.

{\it Non-separating region $[0,t_0^{(\eta)}]$.} During this time period, all the characteristic strips may overlap with each other.
\begin{itemize}
\item{Estimate of $W(t)$:} We first bound
\begin{equation*}
W(t)=\max_i\sup_{(x',t')\atop 0\leq t'\leq t}\big|w_i(x',t')\big|.
\end{equation*}
Using \eqref{eqw}, we have
\begin{equation*}
  \frac{\partial}{\partial s_i}|w_i|\leq \Gamma W^2.
\end{equation*}
By comparing $w_i$ with the solution $Y$ to the following ODE
\begin{equation*}
  \left\{\begin{array}{ll}
  \frac{d}{dt}Y=\Gamma Y^2,\\
  Y(0)=W_0^{(\eta)},
  \end{array}
  \right.
\end{equation*}
we then deduce
\begin{equation}\label{wt0}
  |w_i|\leq Y(t)=\frac{W_0^{(\eta)}}{1-\Gamma W_0^{(\eta)} t}\quad \text{for}\quad t<\min\Big\{\frac1{\Gamma W_0^{(\eta)}},t_0^{(\eta)}\Big\}.
\end{equation}
Noticing that $ t_0^{(\eta)}=O(\eta)$, it yields
\begin{equation}\label{gwt}
  \Gamma W_0^{(\eta)} t_0^{(\eta)}=O(\eta W_0^{(\eta)})=O(\theta).
\end{equation}
Invoking \eqref{gwt} in \eqref{wt0}, for a small parameter $\varepsilon\in(0,\frac1{100}]$ we hence obtain
\begin{equation*}
  |w_i(x,t)|\leq (1+\varepsilon)W_0^{(\eta)}\quad \text{for any}\ x\in\mathbb{R}\ \text{and}\ t\in[0,t_0^{(\eta)}].
\end{equation*}
This further implies that
\begin{equation}\label{y15}
  |W(t)|\leq (1+\varepsilon)W_0^{(\eta)}\quad \text{for}\  t\in[0,t_0^{(\eta)}].
\end{equation}
\item{Estimate of $V(t)$:} Consider the exterior of the characteristic strips $\mathcal{R}_i$. Employing characteristic coordinates, we can label any point $(x',t')\notin\mathcal{R}_i$ by $(z_i',s_i')$ satisfying $z'_i\notin[\eta,2\eta]$ and $w_i^{(\eta)}(z'_i,0)=0$. Integrating  \eqref{eqw} along the characteristic $\mathcal{C}_i$, we then obtain
\begin{equation}
  V(t)=O(\int_0^{t_0^{(\eta)}}w_iw_jds_i)=O(\eta [W(t)]^2)=O(\eta [W_0^{(\eta)}]^2).
\end{equation}
\item{Estimate of $S(t)$:}
By \eqref{eqrho}, the inverse density $\rho_i$ satisfies
\begin{equation} \label{rho est}
  \frac{\partial\rho_i}{\partial s_i}=O(\rho_iW).
\end{equation}
Integrating \eqref{rho est} along the characteristic $\mathcal{C}_i$, we obtain
\begin{equation*}
  \rho_i(z_i,t)=\rho_i(z_i,0)\exp\big(O(tW(t))\big).
\end{equation*}
Together with \eqref{319} and \eqref{y15}, we then derive that
\begin{equation*}
  \rho_i(z_i,t)=\exp\big(O(\eta W_0^{(\eta)})\big)\quad\text{for}\ t\in[0,t_0^{(\eta)}].
\end{equation*}
We can choose $\theta$ to be small enough such that
\begin{equation*}
 1-\varepsilon\leq\rho_i(z_i,t)\leq1+\varepsilon\quad \text{for}\ t\in[0,t_0^{(\eta)}].
\end{equation*}
Here, $\theta$ is a small parameter, that measures the amplitude of the initial data. Therefore, it holds
\begin{equation}\label{y16}
  S(t)=O(1)\quad \text{for}\ t\in[0,t_0^{(\eta)}].
\end{equation}
\item{Estimate of $J(t)$:}
Employing \eqref{eqv}, we have
\begin{equation*}
\frac{\partial v_i}{\partial s_i}=O(S(t)[W(t)]^2).
\end{equation*}
Integrating the above equation along $\mathcal{C}_i$, together with \eqref{y15} and \eqref{y16}, for $t\in[0,t_0^{(\eta)}]$ we obtain
\begin{equation}
  J(t)=O(W_0^{(\eta)}+t[W(t)]^2)=O(W_0^{(\eta)}+\eta [W_0^{(\eta)}]^2)=O(W_0^{(\eta)}).
\end{equation}
\item{Estimate of $\bar{U}(t)$:}
Note that
\begin{equation}\label{phi}
  \Phi(x,t)=\int_{X_5(\eta,t)}^x\frac{\partial \Phi(x',t)}{\partial x} dx'=\int_{X_5(\eta,t)}^x\sum_{k}w_kr_k(x',t)dx'.
\end{equation}
Using the estimate \eqref{y15}, we hence obtain the following estimate of $\bar{U}$,
\begin{equation}
   \bar{U}(t)=\sup_{(x',t')\atop 0\leq t'\leq t}|\Phi(x',t')|=O\big(W(t)(\eta+(\bar{\lambda}_1-\underline{\lambda}_5)t)\big)=O(\eta W_0^{(\eta)}),\quad \forall\ t\in[0,t_0^{(\eta)}].
\end{equation}

\end{itemize}
{\it  Separating region $[t_0^{(\eta)},T]$.} Based on the obtained estimates for $t\in[0,t_0^{(\eta)}]$, we next consider the time interval $t>t_0^{(\eta)}$. And the five characteristic strips $\mathcal{R}_i$ are all separated.
\begin{itemize}
\item{Estimate of $S(t)$:}
For $(x,t)\in \mathcal{R}_i$, by \eqref{eqrho}, we have
\begin{equation}\label{335}
  \frac{\partial \rho_i}{\partial s_i}=O(J+ V S).
\end{equation}
Integrating \eqref{335} along the characteristic $\mathcal{C}_i$, we obtain
\begin{equation*}
\rho_i(z_i, t)=\rho_0(z_i,0)+\int_0^tO(J+ V S)dt'.
\end{equation*}
We hence conclude that
\begin{equation}
  S(t)=O(1+tJ+ tV S).
\end{equation}
\item{Estimate of $J(t)$:}
For $(x,t)\in \mathcal{R}_i$, using \eqref{eqv}, $v_i$ satisfies
\begin{equation} \label{vi}
  \frac{\partial v_i}{\partial s_i}=O(VJ+ V^2 S).
\end{equation}
We then integrate \eqref{vi} and derive that
\begin{equation}\label{548}
  J(t)=O(W_0^{(\eta)}+tVJ+ tV^2 S).
\end{equation}
\item{Estimate of $V(t)$:}
Now we investigate the dynamics of $w_i$ outside the characteristic strip $\mathcal{R}_i$. Via \eqref{eqw}, we have that
\begin{equation}\label{558}
  \frac{\partial w_i}{\partial{s_i}}=O(V^2)+O\Big(\sum_{k\neq i}w_k\Big)V+O\Big(\sum_{m\neq i,k\neq i\atop m\neq k}w_mw_k\Big).
\end{equation}
Note that $\mathcal{C}_i$ issues from $z_i\notin[\eta,2\eta]$ and terminates at $(x,t)\notin\mathcal{R}_i$. If $t'\geq t_0^{(\eta)}$ and $\big(X_i(z_i,t'),t'\big)\in\mathcal{C}_i$, it holds either $\big(X_i(z_i,t'),t'\big)\in\big(\mathbb{R}\times[t_0^{(\eta)},t]\big)\setminus\bigcup_{k}\mathcal{R}_k$ or $\big(X_i(z_i,t'),t'\big)\in\mathcal{R}_k$ for some $k\neq i$.
\begin{center}
\begin{tikzpicture}
\draw(0,0)--(6,0)node[left,above]{$t=0$};
\draw[dashed](0,1.1)--(6,1.1)node[left,above]{$t=t_0^{(\eta)}$};
\node [below]at(3.5,0){$2\eta$};
\node [below]at(2.5,0){$\eta$};

\filldraw [black] (3.5,0) circle [radius=0.01pt]
(4,1) circle [radius=0.01pt]
(4.5,1.5) circle [radius=0.01pt]
(6,3) circle [radius=0.01pt];
\draw [color=black](3.5,0)..controls (4,1) and (4.5,1.5)..(6,3);

\filldraw [black] (2.5,0) circle [radius=0.01pt]
(3,1) circle [radius=0.01pt]
(3.5,1.5) circle [radius=0.01pt]
(5,3) circle [radius=0.01pt];
\draw [color=black](2.5,0)..controls (3,1) and (3.5,1.5)..(5,3);
\node [below] at(5.2,3){$\mathcal{R}_{i}$};

\filldraw [gray] (3.5,0) circle [radius=0.01pt]
(3,2) circle [radius=0.01pt]
(2.2,2.5) circle [radius=0.01pt]
(1.9,3) circle [radius=0.01pt];
\draw [color=gray](3.5,0)..controls (3,2) and (2.2,2.5)..(1.9,3);
\node [below] at(1.8,3){$\mathcal{R}_{k}$};

\filldraw [gray] (2.5,0) circle [radius=0.01pt]
(2,2) circle [radius=0.01pt]
(1.5,2.5) circle [radius=0.01pt]
(1,3) circle [radius=0.01pt];
\draw [color=gray] (2.5,0)..controls (2,2) and (1.5,2.5)..(1,3);

\filldraw [black] (1.3,0) circle [radius=0.01pt]
(1.5,0.7) circle [radius=0.01pt]
(2,1.5) circle [radius=0.01pt]
(3,2.3) circle [radius=0.8pt];
\draw [color=black,thick](1.3,0)..controls (1.5,0.7) and (2,1.5)..(3,2.3);
\node[above]at (3,2.3){$(x,t)$};
\node[below]at (1.3,0){$z_i$};
\end{tikzpicture}
\end{center}
When $(x,t)\notin\mathcal{R}_i$, for ${m\neq i,k\neq i, m\neq k}$, there are only three scenarios: $(x,t)$ stays in $\mathcal{R}_m$, or $(x,t)$ stays in $\mathcal{R}_k$, or $(x,t)$ stays out of all the characteristics. In all of these three cases, the third term $O\Big(\sum_{m\neq i,k\neq i\atop m\neq k}w_mw_k\Big)$ on the right hand side of \eqref{558} can be absorbed by the second term $O\Big(\sum_{k\neq i}w_k\Big)V$. To estimate $O\Big(\sum_{k\neq i}w_k\Big)V$ for $(x,t)\notin\mathcal{R}_i$, we further integrate \eqref{558} along $\mathcal{C}_i$. Noticing that $w_i^{(\eta)}(z_i,0)=0$, we then obtain
\begin{equation}\label{559}
\begin{split}
    w_i(x,t)=&O\Big(tV^2+V\sum_{k\neq i}\int_0^tw_k\big(X_i(z_i,t'),t'\big)dt'\Big)\\
    =&O\Big(tV^2+V\sum_{k\neq i}\int_0^{t_0^{(\eta)}}w_k\big(X_i(z_i,t'),t'\big)dt'\Big)\\
    &+O\Big(V\sum_{k\neq i}\int_{t_0^{(\eta)}}^t w_k\big(X_i(z_i,t'),t'\big)dt'\Big)\\
    =&O\Big(tV^2+\eta[W_0^{(\eta)}]^2+V\sum_{k\neq i}\underbrace{\int_{I_k^i}w_k\big(X_i(z_i,t'),t'\big)dt'}_{M}\Big),
\end{split}
\end{equation}
where we denote $I_k^i:=\{t'\in[t_0^{(\eta)},t]:(x,t')\in\mathcal{C}_i\bigcap\mathcal{R}_k\}$ for $k\neq i$. Here, we use the estimate $V(t)\leq W(t)=O(W_0^{(\eta)})$ for $t\leq t_0^{(\eta)}$ in the second equality. The bi-characteristic coordinates are then employed to control $M$ and we have
\begin{equation}\label{560}
  \begin{split}
    &\int_{I_k^i}w_k\big(X_i(z_i,t'),t'\big)dt'\\
    =&O\Big(\int_{y_k\in[\eta,2\eta]}\big|\frac{\rho_k\big(y_k,t'(y_i,y_k)\big)}{\lambda_i-\lambda_k}w_k\big(y_k,t'(y_i,y_k)\big)\big|dy_k\Big)\\
    =&O(\eta J).
  \end{split}
\end{equation}
Together with \eqref{559} and \eqref{560},  we thus obtain
\begin{equation}\label{355}
  V(t)=O(tV^2+\eta [W_0^{(\eta)}]^2+\eta VJ).
\end{equation}
\item{Estimate of $\bar{U}(t)$:}
We proceed to give the estimate for $\Phi$. If $(x,t)$ does not belong to any characteristic strip, by \eqref{phi}, there holds
\begin{equation}\label{u1}
 \bar{U}(t)=O\big((\eta+(\bar{\lambda}_1-\underline{\lambda}_5)t)V\big).
\end{equation}
If $(x,t)\in\mathcal{R}_k$ for some $k$, we derive
\begin{equation}\label{u2}
\begin{split}
  |\Phi(x,t)|=&\Big|\int_{X_k(\eta,t)}^x\frac{\partial \Phi(x',t)}{\partial x} dx'\Big|=\Big|\int_{X_k(\eta,t)}^x\sum_{k}w_kr_k(x',t)dx'\Big|\\
  \leq&\int_{X_k(\eta,t)}^{X_k(2\eta,t)}|w_k(x',t)|dx'=O\Big(\int_\eta^{2\eta}|w_k(x',t)|\rho_kdz_k\Big)\\
  =&O(\eta J).
  \end{split}
\end{equation}
Here we use the characteristic coordinates. Together with \eqref{u1} and \eqref{u2}, we derive
\begin{equation}
   \bar{U}(t)=O(\eta J+\eta V+\eta tV).
\end{equation}

\end{itemize}

We now finish the proof of the a priori estimates in \eqref{bds}-\eqref{bdu} and have obtained the $L^\infty$ estimates in \eqref{bound} via the aforementioned bootstrap argument.
\subsection{Finite-time blow-up of $w_1$}
We further derive an accurate lower bound and upper bound for $T_{\eta}^*$. When $t$ goes to $T_{\eta}^*$, the first family of characteristics within $\mathcal{R}_1$ collapse to a (first) shock. And the quantity $w_1$ blows up at the shock formation time.

We first impose a bootstrap assumption $v_1<0$, and we will later improve this bound. Recall that the equation for $\rho_1$ is
\begin{equation*}
  \frac{\partial\rho_1}{\partial s_1}=c_{11}^1(\Phi)v_1+O\Big(\sum_{ k\neq 1}w_k\Big)\rho_1.
\end{equation*}
Then noting the fact $c_{11}^1(\Phi)>0$, we have
\begin{equation}\label{381}
  -c_{11}^1|v_1|-\Big|O\Big(\sum_{ k\neq 1}w_k\Big)\Big|\rho_1\leq\frac{\partial\rho_1}{\partial s_1}\leq -c_{11}^1|v_1|+\Big|O\Big(\sum_{ k\neq 1}w_k\Big)\Big|\rho_1.
\end{equation}
Since $|\Phi|=O(\eta W_0^{(\eta)})\leq\delta$, we can choose $\theta$ to be sufficiently small such that
\begin{equation}\label{y37}
 (1-\varepsilon)c_{11}^1(0)\leq c_{11}^1(\Phi)\leq (1+\varepsilon)c_{11}^1(0).
\end{equation}
Employing bi-characteristic coordinates, we then deduce
\begin{equation*}
  \int_0^{t}\sum_{k\neq 1}w_k(X_1(z_1,t'),t')dt'=O(\eta W_0^{(\eta)}+\eta J)=O(\eta W_0^{(\eta)}),
\end{equation*}
which implies
\begin{equation}\label{385}
  1-\varepsilon\leq \exp{\Big(\int_0^{t}O\big(\sum_{k\neq 1}w_k(X_1(z_1,t'),t')\big)dt'\Big)}\leq1+\varepsilon,
\end{equation}
and
\begin{equation}\label{386}
  1-\varepsilon\leq \exp{\Big(-\int_0^{t}O\big(\sum_{k\neq 1}w_k(X_1(z_1,t'),t')\big)dt'\Big)}\leq1+\varepsilon.
\end{equation}
Applying Gr\"{o}nwall's inequality for \eqref{381} and combining with \eqref{y37}-\eqref{386}, we get
\begin{equation}\label{y38}
\begin{split}
  (1-\varepsilon)\Big(1-(1+\varepsilon)^2c_{11}^1(0)\int_0^t|v_1(z,t')|dt'\Big) \leq&\rho_1(z,t)\\
\leq& (1+\varepsilon)\Big(1-(1-\varepsilon)^2c_{11}^1(0)\int_0^t|v_1(z,t')|dt'\Big).
\end{split}
\end{equation}
Next, for $v_1$ we have
\begin{equation}\label{y39}
\frac{\partial v_1}{\partial s_1}=O\Big(\sum_{m\neq 1}w_m\Big)v_1+O\Big(\sum_{m,k\neq1}w_mw_k\Big)\rho_1.
\end{equation}
Integrating \eqref{y39} along $\mathcal{C}_1$, we obtain
\begin{equation*}
\begin{split}
  v_1(z,t)=w_1^{(\eta)}(z,0)+O(tVJ+tV^2S)
  =w_1^{(\eta)}(z,0)+O(\eta [W_0^{(\eta)}]^2).
\end{split}
\end{equation*}
Since $W_0^{(\eta)}=-w_1^{(\eta)}(z_0,0)$, the above inequality implies
\begin{equation*}
-(1+\varepsilon) W_0^{(\eta)}\leq v_1(z_0,t)\leq -(1-\varepsilon) W_0^{(\eta)}.
\end{equation*}
This improves the aforementioned bootstrap assumption $v_1<0$ and it implies
\begin{equation}\label{vz0}
(1-\varepsilon) W_0^{(\eta)}\leq |v_1(z_0,t)|\leq (1+\varepsilon) W_0^{(\eta)}.
\end{equation}
Employing \eqref{y38}, we further obtain that
\begin{equation}\label{711}
  \rho_1(z_0,t)\geq1-(1+\varepsilon)^4c_{11}^1(0)tW_0^{(\eta)}
\end{equation}
and
\begin{equation}
  \rho_1(z_0,t)\leq 1-(1-\varepsilon)^4c_{11}^1(0)tW_0^{(\eta)}.
\end{equation}
From \eqref{711}, we have that $\rho_1(z_0,t)>0$ when
\begin{equation*}
  t<\frac1{(1+\varepsilon)^4c_{11}^1(0)W_0^{(\eta)}}.
\end{equation*}
Meanwhile, we conclude that there exists $T_\eta^*$ (shock formation time) such that
\begin{equation} \label{lim rho1}
  \lim_{t\rightarrow T_\eta^*} \rho_1(z_0,t)=0.
\end{equation}
And $T_\eta^*$ obeys
\begin{equation} \label{Tshock}
 \frac1{(1+\varepsilon)^4c_{11}^1(0)W_0^{(\eta)}}\leq T_\eta^*\leq\frac{1}{(1-\varepsilon)^4c_{11}^1(0)W_0^{(\eta)}}.
\end{equation}
Finally, since $$w_1=\frac{v_1}{\rho_1},$$ as defined in \eqref{defvi}, by \eqref{vz0}, we deduce that 
$$
|w_1|\geq \frac{(1-\varepsilon)W_0^{(\eta)}}{\rho_1}.
$$
Then, using the positivity of $W_0^{(\eta)}$ and \eqref{lim rho1}, we conclude that
$w_1\to \infty$ when $t\to T_\eta^*$ in $\mathcal{R}_1$.

\subsection{Upper bound estimates for other $w_i$ ($i\neq 1$)}
For $w_i$ with $i\neq 1$, even though its evolution equation may be of the Riccati-type, its initial datum is designed to be much smaller compared with $w_1$'s datum. This guarantees that for $i\neq 1$ the evolution of $w_i(z_i,t)$ will not influence the blow-up of $w_1$ as $t\to T_\eta^*$. In particular, when $t\leq T_\eta^*$, we show below $\Big\{\bar W_i=\sup_{(z_i,s_i)\atop z_i\in[\eta,2\eta],\ 0\leq s_i\leq t}w_i(z_i,s_i)\Big\}_{i=2,3,4,5}$ are all bounded.

Invoking the estimate of $V$ into the equations of $w_i$ with $i\neq 1$, for $(x,t)\in\mathcal{R}_i$, we have that
\begin{equation*}
\begin{split}
  \frac{\partial w_i}{\partial{s_i}}=&-c^i_{ii}w_i^2+O\Big(\sum_{k\neq i}w_k\Big)w_i+\underbrace{O\Big(\sum_{m\neq i,k\neq i\atop m\neq k}w_mw_k\Big)}_{\text{containing weak interaction}}\\
  =&-c^i_{ii}w_i^2+O(V)w_i+O(V^2)\\
  =&-c^i_{ii}w_i^2+O\big(\eta[W_0^{(\eta)}]^2\big)w_i+O\big(\eta^2[W_0^{(\eta)}]^4\big).
  \end{split}
\end{equation*}
Equivalently, we rewrite the above equation as
\begin{small}
\begin{equation} \label{w7ode}
\frac{d}{ds}\Big[\exp\Big(O\big(\eta[W_0^{(\eta)}]^2\big)s\Big)w_i\Big]=-\exp\Big(O\big(\eta[W_0^{(\eta)}]^2\big)s\Big)c^i_{ii}w_i^2+\exp\Big(O\big(\eta[W_0^{(\eta)}]^2\big)s\Big)O\big(\eta^2[W_0^{(\eta)}]^4\big).
\end{equation}
\end{small}
Using \eqref{Tshock}, we then have
\begin{equation*}
\exp\Big(O\big(\eta[W_0^{(\eta)}]^2\big)s\Big)\leq \exp\Big(O\big(\eta[W_0^{(\eta)}]^2\big)T_\eta^*\Big)\leq \exp\Big(O\big(\eta W_0^{(\eta)}\big)\Big)=O(e^\theta).
\end{equation*}
Here, $\theta$ and $\eta$ are small parameters, which characterize the magnitude of the initial data and the size of their supports, respectively. Then, by choosing $\theta$ to be sufficiently small\footnote{In the shock formation argument of this section, $\theta$ may be taken independent of $\eta$. However, the $\eta$-dependence of $\theta$ will emerge and play a crucial role in the subsequent ill-posedness analysis in Section \ref{ill}.}, we further obtain
\begin{equation} \label{theta3}
1-\varepsilon\leq\exp\Big(O\big(\eta[W_0^{(\eta)}]^2\big)s\Big)\leq 1+\varepsilon.
\end{equation}
Thus, back to \eqref{w7ode}, it holds that
\begin{equation}\label{deqw7}
\begin{split}
\frac{d}{ds}\Big[\exp\Big(O\big(\eta[W_0^{(\eta)}]^2\big)s\Big)w_i\Big]
\leq Cw_i^2+O\big(\eta^2[W_0^{(\eta)}]^4\big).
\end{split}
\end{equation}
Integrating \eqref{deqw7} along $\mathcal{C}_i$, we get
\begin{equation}\label{621}
\begin{split}
\bar W_i&\leq O\big(w_i(z,0)+t\bar W_i^2+t\eta^2[W_0^{(\eta)}]^4\big)\\
&\leq O\big([W_0^{(\eta)}]^2+t\bar W_i^2+\eta^2[W_0^{(\eta)}]^3\big)\\
&\leq O\big([W_0^{(\eta)}]^2+t\bar W_i^2\big).
\end{split}
\end{equation}
Now we introduce an additional bootstrap assumption
\begin{equation} \label{boottcw}
t\bar W_i\leq \theta^{\frac12}.
\end{equation}
Then by \eqref{621}, it holds
\begin{equation} \label{Wcc7}
\bar W_i\leq  O\big([W_0^{(\eta)}]^2\big).
\end{equation}
And by \eqref{Tshock}, we get
\begin{equation}
t\bar W_i\leq  O(W_0^{(\eta)})=O(\theta)< \theta^{\frac12}.
\end{equation}
This improves the bootstrap assumption \eqref{boottcw} and gives the desired upper bound for $|w_i|$ with $i\neq1$.

\section{The ill-posedness and norm inflation}\label{ill}
We are ready to prove  Theorem \ref{2D}. Our ill-posedness result is twofold.

The first part is the instantaneous shock formation. This immediately follows from our construction \eqref{dataw0}\eqref{data2} for initial data and our estimate for the shock formation time \eqref{Tshock}. In particular, we obtain that, as $\eta \to 0$, the shock formation time $T^*_\eta \to 0 $ and the instantaneous shock forms.
\begin{remark} \label{ambient}
Note that the general 2D ideal compressible MHD system \eqref{MHD} forms a hyperbolic system of transport equations. In particular, by taking divergence of the third equation in \eqref{MHD}, one could see that the divergence-free condition $\nabla\cdot H=0$ can be transported all the way as long as this condition holds for the initial data. The governing maximal speed of this transport system would be the speed of the fast magnetosonic wave.
To apply our shock formation result in Section \ref{aprior}, we employ the finite-propagating-speed property (see Appendix \ref{FSP}), i.e., the maximal speed of the fast magnetosonic wave is finite.
\end{remark}

In this section, we also prove the second part of our result, which is the $H^1$ norm inflation. More precisely, within the planar symmetric region, we show that the $H^1$ norm of the solutions to 2D ideal compressible MHD system \eqref{MHD} blows up at the shock formation time $T_{\eta}^*$ in a future domain $\Omega_{T_\eta^*}$.

According to our previous analysis, the propagation-speed is at most $\lambda_1$. Hence, without loss of generality, we restrict our focus to the following region
\begin{equation} \label{init domain}
\Omega_0=\{(x,y_2):\left(x-\lambda_1(0) T_\eta^*-2\eta\right)^2+y_2^2 \leq  {\left(\lambda_1(0) T_\eta^*+\eta \right)}^2\}.
\end{equation}
For this particular choice, note that we get inspired by the following fact: according to finite speed of propagation (Appendix \ref{FSP}), the characteristic $\mathcal{C}_1(2\eta)$ issuing from $x=2\eta$ is a straight line with slope $\lambda_1(0)$. Considering the initial data constructed in Section \ref{data}, we note that
\begin{equation*}
 \Omega_0 \subseteq \Big\{(x,y_2): \psi\Big(\frac{|\ln(x)|^\delta y_2}{\sqrt{x}}\Big)=1,\ \text{if}\ x\leq 2\eta;\ \mathcal{X}\Big(\frac{x}{\eta}\Big)=0,\ \text{if}\ x>2\eta\Big\}
\end{equation*}
with $\psi\Big(\frac{|\ln(x)|^\delta y_2}{\sqrt{x}}\Big)$ and $\mathcal{X}\Big(\frac{x}{\eta}\Big)$ being cut-off functions given in \eqref{dataw0}. Therefore, we require the constructed initial data to be planar symmetric within $\Omega_0$, i.e., $\tilde{w}_i^{(\eta)}(x,y_2)|_{\Omega_0}=\tilde{w}_i^{(\eta)}(x)$ for $i=1,2,3,4,5$. Moreover, as the initial data are compactly supported within $[\eta,2\eta]$ in $x$-direction, we have that
$$T_\eta^*\sim\frac{1}{\theta|\ln \eta|^\alpha}.$$
Together with \eqref{init domain}, this further implies that, within this non-trivial region, the $x_2$-scale would be $O(\frac{\sqrt{\eta}}{\theta^\frac12|\ln \eta|^\frac{\alpha}{2}})$,
and it should satisfy
\begin{equation}
y_2^2\lesssim 2\eta T_\eta^*\sim \frac{\eta}{\theta |\ln\eta|^\alpha}\leq x_2^2\leq x|\ln x|^{-2\delta}\sim \eta|\ln \eta|^{-2\delta}.
\end{equation}
Therefore, we need to choose $\theta$ such that it satisfies
\begin{equation}
\theta\gtrsim|\ln\eta|^{2\delta-\alpha}   \quad \text{with}\quad 2\delta-\alpha<0.
\end{equation}

Considering the domain of dependence $\mathcal{D}(\Omega_0;T_\eta^*)$, we denote its $T_\eta^*$-slice as $\Omega_{T_\eta^*}$. Due to the finite speed of propagation (see Appendix \ref{FSP}), the information within $\Omega_{T_\eta^*}$ can be completely determined by the initial data in domain $\Omega_0$. See Figure \ref{plane sym} below. Based on the above construction, we can apply the planar symmetric shock formation argument within $\mathcal{D}(\Omega_0;T_\eta^*)$.
\begin{figure}[H]
\centering

\tikzset{every picture/.style={line width=0.75pt}} 

\begin{tikzpicture}[x=0.75pt,y=0.75pt,yscale=-1,xscale=1]
\filldraw[white, fill=gray!40] (219.32,243) .. controls (259.32,213.1) and (257,176) .. (297,146)--(389,146)--(463,243)--(219.32,243);
\filldraw[gray!40, fill=gray!80] (219.32,243) .. controls (259.32,213.1) and (257,176) .. (297,146)--(343,146)--(281,243)--(219.32,243);
\draw    (219.32,243) .. controls (259.32,213.1) and (257,176) .. (297,146) ;
\draw [shift={(297,146)}, rotate = 323.13] [color={rgb, 255:red, 0; green, 0; blue, 0 }  ][fill={rgb, 255:red, 0; green, 0; blue, 0 }  ][line width=0.75]      (0, 0) circle [x radius= 3.35, y radius= 3.35]   ;
\draw    (463,243) -- (219.32,243) ;
\draw [shift={(219.32,243)}, rotate = 179.51] [color={rgb, 255:red, 0; green, 0; blue, 0 }  ][fill={rgb, 255:red, 0; green, 0; blue, 0 }  ][line width=0.75]      (0, 0) circle [x radius= 3.35, y radius= 3.35]   ;
\draw    (389,146) -- (297,146) ;
\draw    (463,243) -- (389,146) ;
\draw [fill={rgb, 255:red, 155; green, 155; blue, 155 }  ,fill opacity=1 ] [dash pattern={on 4.5pt off 4.5pt}]  (281,243) -- (343,146) ;
\draw [shift={(343,146)}, rotate = 302.86] [color={rgb, 255:red, 0; green, 0; blue, 0 }  ][fill={rgb, 255:red, 0; green, 0; blue, 0 }  ][line width=0.75]      (0, 0) circle [x radius= 3.35, y radius= 3.35]   ;
\draw [shift={(281,243)}, rotate = 302.86] [color={rgb, 255:red, 0; green, 0; blue, 0 }  ][fill={rgb, 255:red, 0; green, 0; blue, 0 }  ][line width=0.75]      (0, 0) circle [x radius= 3.35, y radius= 3.35]   ;
\draw    (480,186) -- (441,186) ;
\draw [shift={(438,186)}, rotate = 360] [fill={rgb, 255:red, 0; green, 0; blue, 0 }  ][line width=0.08]  [draw opacity=0] (8.93,-4.29) -- (0,0) -- (8.93,4.29) -- cycle    ;

\draw (337,251.4) node [anchor=north west][inner sep=0.75pt]    {$\Omega _{0}$};
\draw (333,110.4) node [anchor=north west][inner sep=0.75pt]    {$\Omega _{T_{\eta }^{*}}$};
\draw (486,174.4) node [anchor=north west][inner sep=0.75pt]    {$\mathcal{D}(\Omega _{0};T_\eta^*)$};
\draw (342,184.4) node [anchor=north west][inner sep=0.75pt]    {$\Phi \equiv 0$};
\draw (210,249.4) node [anchor=north west][inner sep=0.75pt]    {$\eta $};
\draw (266,249.4) node [anchor=north west][inner sep=0.75pt]    {$2\eta $};

\end{tikzpicture}

\caption{\small{\bf Domain of dependence.} In this picture, the domain of dependence $\mathcal{D}(\Omega_0;T_\eta^*)$ corresponds to the whole colored trapezoid region. According to the finite speed of propagation, within $\mathcal{D}(\Omega_0;T_\eta^*)$, the solutions, as well as the initial data, are compactly supported in the dark gray region. While, the light gray region denotes the trivial part. Hence the dashed line in this figure is straight.}
\label{plane sym}
\end{figure}

We also point out for $\Phi\in B_\delta^5(0)$, $\Omega_{T_\eta^*}$ is uniformly close to the following 2D-disk
\begin{equation} \label{region}
\Omega_{T_\eta^*}\approx\{(x,y_2,T_\eta^*): x=X_1(z,T_\eta^*),\ (x-\lambda_1(0)T_\eta^*-2\eta)^2+y_2^2\leq\eta^2,  \  \text{for}\ \eta\leq z\leq2\eta \}.
\end{equation}
This implies that
\begin{equation}
\int_{(x,y_2)\in\Omega_{T_\eta^*}} dy_2 \approx(1+O(\varepsilon))\sqrt{\big(z-\eta+O(\varepsilon)\eta\big)\big(3\eta-z+O(\varepsilon)\eta\big)}.
\end{equation}

Before showing the blow-up of the solution's $H^1$-norm, we first derive an upper bound estimate for $|\partial_{z_1}\rho_1(z_1,s_1)|$.
\begin{prop} \label{dzrho1}
There exists a uniform constant $C$ depending only on $\varepsilon,\theta$ and $\eta$, such that for any $s_1\leq T_\eta^*$, there holds $|\partial_{z_1}\rho_1(z_1,s_1)|\leq C.$
\end{prop}

\begin{proof}
This proposition can be proven in the same fashion as we did for the 3D case in \cite{an,an2}. We outline the main steps here. Since the coordinate transformation obeys \eqref{biytoz}, it holds
\begin{equation}\label{rhod}
 \partial_{z_1}\rho_1=\partial_{y_1}\rho_1+\frac{\rho_1}{2C_f}\partial_{s_1}\rho_1.
\end{equation}
To control $\partial_{z_1}\rho_1$, we first estimate $\partial_{y_1}\rho_1$. Denote
\begin{equation*}
\tau_{1} ^{(5)}:=\partial_{y_1}\rho_1,\quad\pi_{1}^{(5)}:=\partial_{y_1}v_1.
\end{equation*}
Since by \eqref{biytoz}, it holds that
\begin{equation*}
  \partial_{y_5}=\frac{\rho_5}{\lambda_1-\lambda_5}\partial_{s_1}=\frac{\rho_5}{2C_f}\partial_{s_1},
\end{equation*}
we have that $\tau_{1}^{(5)}$ obeys
\begin{equation}\label{97}
\begin{split}
  \partial_{y_5}\tau_{1}^{(5)}=&\partial_{y_1}\partial_{y_5}\rho_1=\partial_{y_1}\Big(\frac{\rho_5}{2C_f}\partial_{s_1}\rho_1\Big)=\partial_{y_1}\Big(\frac{\rho_1\rho_5}{2C_f}\sum_m c_{1m}^1w_m\Big)\\
  =&\frac{\rho_5}{2C_f}\Big(c_{11}^1\partial_{y_1}v_1+\sum_{m\neq1}c_{1m}^1w_m \partial_{y_1}\rho_1\Big)+\frac{\rho_1}{2C_f}\Big(\sum_{m\neq4}c_{1m}^1w_m\partial_{y_1}\rho_5+c_{15}^1\partial_{y_1}v_5\Big)\\
  &+\frac{\rho_1\rho_5}{2C_f}\Big(\sum_{m=2,3,4}c_{1m}^1\frac1{\rho_m}\partial_{y_1}v_m-\sum_{m=2,3,4}c_{1m}^1\frac{w_m}{\rho_m}\partial_{y_1}\rho_m\Big)\\
  &-\frac{\partial_{y_1}C_f}{2C_f^2}\rho_1\rho_5\sum_m c_{1m}^1w_m+\frac{\rho_1\rho_5}{2C_f}\sum_m \partial_{y_1}c_{1m}^1w_m.
\end{split}
\end{equation}
In \eqref{97}, the quantities remaining to be estimated are: $\partial_{y_1}\lambda_1,\partial_{y_1}c_{1m}^1,\partial_{y_1}\rho_m$ and $\partial_{y_1}v_m$. Firstly, for $\partial_{y_1}\lambda_1$, using bi-characteristic coordinates $(y_1,y_i)$ ($i\neq1$), it satisfies
\begin{equation}\label{9.11}
  \begin{split}
    \partial_{y_1}\lambda_1=&\nabla_\Phi\lambda_1\cdot\partial_{y_1}\Phi=\nabla_\Phi\lambda_1\cdot[\partial_{s_i}X_i\partial_{y_1}t'\partial_{x}\Phi+\partial_{y_1}t'\partial_{t}\Phi]\\
    =&\nabla_\Phi\lambda_1\cdot\Big[\lambda_i\frac{\rho_1}{\lambda_i-\lambda_1}\sum_kw_kr_k+\frac{\rho_1}{\lambda_i-\lambda_1}(-A(\Phi)\sum_kw_kr_k)\Big]\\
    =&O(v_1+\rho_1\sum_{k\neq1}w_k).
  \end{split}
\end{equation}
Similarly, we get
\begin{equation}\label{9.12}
  \begin{split}
    \partial_{y_1}c_{1m}^1=&\nabla_\Phi c_{1m}^1\cdot\partial_{y_1}\Phi=\nabla_\Phi c_{1m}^1\cdot[\partial_{s_i}X_i\partial_{y_1}t'\partial_{x}\Phi+\partial_{y_1}t'\partial_{t}\Phi]\\
    =&\nabla_\Phi c_{1m}^1\cdot\Big[\lambda_i\frac{\rho_1}{\lambda_i-\lambda_1}\sum_kw_kr_k+\frac{\rho_1}{\lambda_i-\lambda_1}(-A(\Phi)\sum_kw_kr_k)\Big]\\
    =&O\Big(v_1+\rho_1\sum_{k\neq1}w_k\Big).
  \end{split}
\end{equation}
Next, by \eqref{eqrho}-\eqref{eqv}, we have
\begin{equation}\label{9.14}
  \partial_{y_1}\rho_m=\frac{\rho_1}{\lambda_m-\lambda_1}\partial_{s_m}\rho_m=O\Big(\rho_mv_1+\rho_1\rho_m\sum_{k\neq 1}w_k\Big) \quad \text{when}\quad m\neq 1,
\end{equation}
and
\begin{equation}\label{9.15}
  \partial_{y_1}v_m=\frac{\rho_1}{\lambda_m-\lambda_1}\partial_{s_m}v_m=O\Big(\rho_mv_1\sum_{k\neq1}w_k+\rho_1\rho_m\sum_{j\neq1,k\neq1\atop j\neq k}w_jw_k\Big)\quad \text{when}\quad m\neq 1.
\end{equation}
Therefore, by \eqref{9.11}-\eqref{9.15}, one can rewrite \eqref{97} as
\begin{equation}\label{linear}
\begin{split}
  \partial_{y_5}\tau_{1}^{(5)}:=B_{11}\tau_1^{(5)}+B_{12}\pi_1^{(5)}+B_{13}
\end{split}
\end{equation}
with $B_{11},B_{12},B_{13}$ being uniform constants depending on $\eta$.

Next, we deduce the following evolution equation for $\pi_1^{(5)}$ in the same manner
\begin{equation}\label{917}
  \begin{split}
  \partial_{y_5}\pi_{1}^{(5)}
  =&\frac{\rho_5 }{2C_f}\Big(\sum_{p\neq 1,q\neq 1\atop p\neq q}\gamma_{pq}^1w_pw_q \tau_{1}^{(5)}+\sum_{p\neq 1}\gamma_{1p}^1w_p\pi_{1}^{(5)}\Big)\\
  &-\frac{\partial_{y_1}C_f}{2C_f^2}\Big(\sum_{p\neq 1}\gamma_{1p}^1 w_p v_1\rho_5+\sum_{p\neq 1,q\neq 1\atop p\neq q}\gamma_{pq}^1w_pw_q\rho_5\rho_1\Big)\\
    &+\frac{\rho_5\rho_1}{2C_f}\Big(\sum_{p\neq 1}\partial_{y_1}\gamma_{1p}^1 w_p w_m+\sum_{p\neq 1,q\neq 1\atop p\neq q}\partial_{y_1}\gamma_{pq}^1w_pw_q\Big)\\
    &+\frac{\rho_1}{2C_f}\Big(\sum_{p\neq 1}\gamma_{1p}^1 w_p w_1+\sum_{p\neq 1,q\neq 1\atop p\neq q}\gamma_{pq}^1w_pw_q\Big)\partial_{y_1}\rho_5\\
    &+\frac{\rho_5\rho_1}{2C_f}\Big(\sum_{p=2,3,4}\gamma_{1p}^1\frac{w_1}{\rho_p}+\sum_{p\neq 1,q\neq 1\atop p\neq q}\gamma_{pq}^1\frac{w_q}{\rho_p}\Big)(\partial_{y_1}v_p-w_p\partial_{y_1}\rho_p)\\
    &+\frac{\rho_5\rho_1}{2C_f }\sum_{p\neq 1,q\neq 1\atop p\neq q}\gamma_{pq}^1\frac{w_p}{\rho_q}(\partial_{y_1}v_q-w_q\partial_{y_1}\rho_q)
    +\frac{\rho_1}{2C_f }\gamma_{15}^1 w_1\partial_{y_1}v_5\\
    :=&B_{21}\tau_1^{(5)}+B_{22}\pi_1^{(5)}+B_{23}.
  \end{split}
\end{equation}
Similarly to \eqref{linear}, the constants $B_{21},B_{22},B_{23}$ here are also uniformly bounded.

Then, we bound the initial data of $\tau_1^{(5)}$ and $\pi_1^{(5)}$. It follows from \eqref{biytoz} and \eqref{319} that
\begin{equation*}
  \begin{split}
    \tau_1^{(5)}(z_1,0)=&\partial_{z_1}\rho_1(z_1,0)-\frac{\rho_1(z_1,0)}{2C_f}\partial_{s_1}\rho_1(z_1,0)\\
     =&-\frac{1}{2C_f}\sum_{k}c_{1k}^1 w_k(z_1,0)=O(W_0^{(\eta)})<+\infty.
  \end{split}
\end{equation*}
And since $v_1^{(\eta)}(z_1,0)=w_1^{(\eta)}(z_1,0)$, we similarly obtain
\begin{equation*}
  \begin{split}
    \pi_1^{(5)}(z_1,0)=&\partial_{z_1}v_1(z_1,0)-\frac{\rho_1(z_1,0)}{2C_f}\partial_{s_1}v_1(z_1,0)\\
    =&\partial_{z_1}w_1(z_1,0)-\frac{1}{2C_f}\sum_{q\neq 1,q\neq p}\gamma_{pq}^1w_p(z_1,0)w_q(z_1,0)\rho_1(z_1,0)\\
    =&O(\partial_{z_1}w_1(z_1,0)+[W_0^{(\eta)}]^2)<+\infty.
  \end{split}
\end{equation*}
Now by applying Gr\"onwall's inequality to \eqref{linear} and \eqref{917}, for $s_1\leq T_\eta^*$ we deduce that $\tau_{1}^{(5)}:=\partial_{y_1}\rho_1(z_1,s_1)$ is bounded on $z_1\in[\eta,2\eta]$.

Back to \eqref{rhod}, invoking all these estimates, we hence prove
\begin{equation*}
  \partial_{z_1}\rho_1=\partial_{y_1}\rho_1+O(v_1+\sum_{m\neq 1}w_m\rho_1).
\end{equation*}
Employing the estimates for $J(t)$, $S(t)$ and $V(t)$ obtained in Section \ref{aprior}, consequently, we conclude that $|\partial_{z_1}\rho_1|$ is uniformly bounded by a uniform constant $C$. \end{proof}

Next, we begin to estimate the $H^1$ norm of the solution. The integration is taken in a subinterval $(z_0,z_0^*]\subseteq[\eta,2\eta]$ where $|w_1^{(\eta)}(z,0)|>\frac12W_0^{(\eta)}$ for $z\in (z_0,z_0^*]$. In this subinterval, it follows from \eqref{vz0} that $|v_1(z,t)|$ admits a lower bound $|v_1(z,t)|\geq \frac14W_0^{(\eta)}$. Then, together with the shock formation at $(z_0,T_\eta^*)$, i.e., $\lim\limits_{t\to T_\eta^{*}}\rho_1(z_0,t)=0$ and Proposition \ref{dzrho1}, we obtain
\begin{equation*}
\begin{split}
  &\lim\limits_{t\to T_\eta^{*}}\|w_1(\cdot,t)\|_{L^2(\Omega_{t})}^2\\
  \geq &C\lim\limits_{t\to T_\eta^{*}}\int_{\eta}^{2\eta}\Big|\frac{v_1}{\rho_1}(z,t)\Big|^2\rho_1(z,t) \sqrt{\big(z-\eta+O(\varepsilon)\eta\big)\big(3\eta-z+O(\varepsilon)\eta\big)}dz\\
\geq& C\sqrt{(z_0-\eta)(3\eta-z_0^*)}[W_0^{(\eta)}]^2\lim\limits_{t\to T_\eta^{*}}\int_{z_0}^{z_0^*}\frac{1}{\rho_1(z,t)-\rho_1(z_0,t)}dz\\
\geq& C\sqrt{(z_0-\eta)(3\eta-z_0^*)}[W_0^{(\eta)}]^2\lim\limits_{t\to T_\eta^{*}}\int_{z_0}^{z_0^*}\frac{1}{(\sup_{z\in(z_0,z_0^*]}|\partial_z\rho_1(z,t)|)(z-z_0)}dz\\
\geq& C_\eta\int_{z_0}^{z_0^*}\frac{1}{z-z_0}dz=+\infty.
  \end{split}
\end{equation*}
Here we crucially use the property $\lim\limits_{t\to T_\eta^{*}}\rho_1(z_0,t)=0$ of shock formation and the boundedness of $\partial_{z_1}\rho_1$ in Proposition \ref{dzrho1}. Thus, the proved ill-posedness is driven by instantaneous shock formation.

Finally, we go back to the original 2D ideal MHD system \eqref{MHD}. Noting that since the region $\Omega_{t}$ stays in $\mathcal{R}_1$, therefore $w_i\big|_{\Omega_{t}}$ are controlled by $V(t)$ for $i=2,3,4,5$. Combining this with the fact that $|r_k|=O(1)$, we hence obtain
\begin{equation*}
  \begin{split}
    \lim\limits_{t\to T_\eta^{*}}\|\partial_xu_1\|_{L^2(\Omega_{t})}=&\lim\limits_{t\to T_\eta^{*}}\|\sum_{k=1}^5 w_k r_{k1}\|_{L^2(\Omega_{t})}\geq C\lim\limits_{t\to T_\eta^{*}}\Big(\|w_1\|_{L^2(\Omega_{t})}-\sum_{k=2,4,5}\|w_k \|_{L^2(\Omega_{t})}\Big)\\   \geq&C\lim\limits_{t\to T_\eta^{*}}\Big(|w_1\|_{L^2(\Omega_{t})}-3V( t)|\Omega_{t}|^{\frac12}\Big)\\
    \geq&C\lim\limits_{t\to T_\eta^{*}}\Big(\|w_1\|_{L^2(\Omega_t)}-3\eta^2W_0^{(\eta)}\Big)= +\infty
  \end{split}
\end{equation*}
and
\begin{equation*}
  \begin{split}
    \lim\limits_{t\to T_\eta^{*}}\|\partial_x \varrho\|_{L^2(\Omega_t)}=&\lim\limits_{t\to T_\eta^{*}}\|\sum_{k=1}^5 w_k r_{k3}\|_{L^2(\Omega_{t})}\geq C\lim\limits_{t\to T_\eta^{*}}\Big(\|w_1\|_{L^2(\Omega_{t})}-\sum_{k=2,3,4,5}\|w_k \|_{L^2(\Omega_{t})}\Big)\\
    \geq&C\lim\limits_{t\to T_\eta^{*}}\Big(\|w_1\|_{L^2(\Omega_{t})}-4\eta^2W_0^{(\eta)}\Big)= +\infty.
  \end{split}
\end{equation*}
This concludes the proof of Theorem \ref{2D} when $\kappa\neq 0$.

\section{Proof for $\kappa=0$ case}\label{h10}
In this section, we deal with the case of $\kappa=0$. We first explore the detailed structures of the decomposed system for this scenario. Then we prove the shock formation and the desired ill-posedness. Compared with the $\kappa\neq 0$ case, the system here is non-strictly hyperbolic. Nonetheless, as will be explained below, since the non-separating characteristic speeds are identical, we can treat the corresponding characteristics as in the same family. The arguments in Section \ref{aprior} and Section \ref{ill} then apply to the $\kappa=0$ case.

Under planar symmetry, when $H_1=\kappa=0$, $\Phi=(u_1,u_2,\varrho-1, H_2,S)^T$ satisfies:
\begin{equation}\label{eq0}
\partial_t\Phi+A(\Phi)\partial_x\Phi=0
\end{equation}
with the coefficient matrix being
\begin{equation} \label{0MHD}
A(\Phi)=\begin{pmatrix}
u_1 & 0 & c^2/\varrho & \mu_0H_2/\varrho &  c^2/\gamma \\
0 & u_1 & 0 & 0 &0 \\
\varrho & 0 & u_1 & 0 & 0 \\
H_2 & 0 & 0 &u_1 &  0 \\
0 & 0 & 0 & 0 &  u_1
\end{pmatrix}.
\end{equation}
The eigenvalues of \eqref{0MHD} are
\begin{equation} \label{egv0}
\begin{split}
&\lambda_1=u_1+C_f,\quad \lambda_2=\lambda_3=\lambda_4=u_1,\quad \lambda_5=u_1-C_f,
\end{split}
\end{equation}
where $C_f$ is defined by
$$C_f=\sqrt{\frac{\mu_0H_2^2}{\varrho}+c^2}.$$
And the eigenvalues satisfy
\begin{equation*}
\begin{split}
\lambda_5<\lambda_4=\lambda_3=\lambda_2<\lambda_1.
\end{split}
\end{equation*}
Since $A(\Phi)$ admits an eigenvalue of triple multiplicity, the system \eqref{eq0} is not strictly hyperbolic. We choose the right eigenvectors as:
{\begin{equation} \label{rvec0}
r_1=\left(\begin{array}{cc} \frac{C_f}{\varrho}\\0\\1\\ \frac{H_2}{\varrho}\\0\end{array}\right),
r_2=\left(\begin{array}{cc}0\\1\\0\\0\\0\end{array}\right),
r_3=\left(\begin{array}{cc}0\\0\\1\\0\\ -\frac{\gamma}{\varrho}\end{array}\right),
r_4=\left(\begin{array}{cc}0\\0\\0\\1\\-\frac{\gamma\mu_0H_2}{\varrho c^2}\end{array}\right),
r_5=\left(\begin{array}{cc} -\frac{C_f}{\varrho}\\0\\1\\\frac{H_2}{\varrho}\\0\end{array}\right).
\end{equation}}
And the left eigenvectors are set to be the dual of the right ones:
\begin{equation} \label{dual}
l_ir_j=\delta_{ij}.
\end{equation}
Via calculations, we construct the following left eigenvectors
\begin{equation}
\begin{split}
&l_1=\Big(\frac{\varrho}{2C_f},0,\frac{c^2}{2 C_f^2},\frac{\mu_0H_2}{2 C_f^2},\frac{\varrho c^2}{2\gamma C_f^2}\Big),\\
& l_2=(0,1,0,0,0),\\
& l_3=\Big(0,0,\frac{\mu_0 H_2^2}{\varrho C_f^2},-\frac{\mu_0 H_2}{C_f^2},-\frac{\varrho c^2}{\gamma C_f^2}\Big),\\
&l_4=\Big(0,0,-\frac{c^2H_2}{\varrho C_f^2},\frac{c^2}{ C_f^2},-\frac{ c^2 H_2}{\gamma C_f^2}\Big)\\
&l_5=\Big(-\frac{\varrho}{2C_f},0,\frac{c^2}{2 C_f^2},\frac{\mu_0H_2}{2 C_f^2},\frac{\varrho c^2}{2\gamma C_f^2}\Big),
\end{split}
\end{equation}
which satisfies \eqref{dual}. With these eigenvectors, we further decompose $\partial_x \Phi$ as in Section \ref{decom} and obtain:
\begin{align}
  \partial_{s_i}\rho_i=&c_{ii}^iv_i+\Big(\sum_{m\neq i}c_{im}^iw_m\Big)\rho_i,\label{eqrho0}\\
  \partial_{s_i}w_i=&-c_{ii}^iw_i^2+\Big(\sum_{m\neq i}(-c_{im}^i+\gamma_{im}^i)w_m\Big)w_i+\sum_{m\neq i,k\neq i\atop m\neq k}\gamma_{km}^iw_kw_m,\label{eqw0}\\
  \partial_{s_i}v_i=&\Big(\sum_{m\neq i}\gamma_{im}^iw_m\Big)v_i+\sum_{m\neq i,k\neq i\atop m\neq k}\gamma_{km}^iw_kw_m\rho_i,\label{eqv0}
\end{align}
where
 \begin{align*}
&c_{im}^i=\nabla_\Phi\lambda_i\cdot r_m,\\
  &\gamma_{im}^i=-(\lambda_i-\lambda_m)l_i \cdot(\nabla_\Phi r_i \cdot r_m-\nabla_\Phi r_m \cdot r_i),\quad m\neq i,\\
  &\gamma_{km}^i=-(\lambda_k-\lambda_m)l_i \cdot (\nabla_\Phi r_k \cdot r_m), \qquad\qquad\qquad k\neq i,\  m\neq i.
\end{align*}
Since $H_2$ is small around zero, one can check that $C_f\approx c$. Thus $r_i$, $\nabla_\Phi r_i$ and $l_i$ are of order $O(1)$. Hence the coefficients $c_{im}^i,\gamma_{im}^i,\gamma_{km}^i$ are all uniformly bounded by $O(1)$. Especially, we have
\begin{equation*}
\begin{split}
c_{11}^1(0)=&\nabla_\Phi\lambda_1(0)\cdot r_1(0)\\
=&\Big(1,0,\frac{(\gamma-1)\sqrt{A\gamma}}{2},0,A\gamma\Big)\cdot (\sqrt{A\gamma},0,1,0,0)^T\\
=&\frac{(\gamma+1)\sqrt{A\gamma}}{2}>0.
\end{split}
\end{equation*}
The first family of characteristic is now genuinely nonlinear. The main difference between the case of $H_1=0$ and $H_1\neq0$ is the non-strict hyperbolicity. For $\lambda_2=\lambda_3=\lambda_4$, the corresponding characteristic strips $\mathcal{R}_2,\mathcal{R}_3,\mathcal{R}_4$ are the same. Our strategy is to consider characteristic waves that propagate in and out of the following three characteristic strips: $\{\mathcal{R}_1,\mathcal{R}_{\bar2},\mathcal{R}_5\}$, where $\mathcal{R}_{\bar2}:=\mathcal{R}_{2}=\mathcal{R}_{3}=\mathcal{R}_{4}$. These three strips will be completely separated when $t>t_0^{(\eta)}$. See the following figure.
\begin{figure}[H]
\centering
\begin{tikzpicture}[fill opacity=0.5, draw opacity=1, text opacity=1]
\node [below]at(3.5,0){$2\eta$};
\node [below]at(2.5,0){$\eta$};

\filldraw[white, fill=gray!40](3.5,0)..controls (3,1) and (1.8,2.6)..(1,3.5)--(3.2,3.5)..controls (2.1,1.5) and (2.6,1)..(2.5,0);
\filldraw[white, fill=gray!40](3.5,0)..controls (3.4,1) and (3.2,1.5)..(4.8,3.5)--(6.5,3.5)..controls (3.8,1.5) and (3,1)..(2.5,0);

\filldraw[white,fill=gray!80](2.5,0)..controls (3,1) and (3.8,1.5)..(6.5,3.5)--(8.2,3.5)..controls (5,1.5) and (4.3,1)..(3.5,0);
\filldraw[white,fill=gray!80](2.5,0)..controls (2.6,1) and (2.1,1.5)..(3.2,3.5)--(4.8,3.5)..controls (3.2,1.5) and (3.4,1)..(3.5,0);
\filldraw[white,fill=gray!80](2.5,0)..controls (1.2,2.2) and (0.5,2.8)..(-0.5,3.5)--(1,3.5)..controls (1.8,2.6) and (3,1)..(3.5,0);

\draw[->](-0.5,0)--(8.2,0)node[left,below]{$t=0$};
\draw[dashed](-0.5,1.7)--(8.2,1.7)node[right,below]{$t=t_0^{(\eta)}$};

\draw [color=black](3.5,0)..controls (4.3,1) and (5,1.5)..(8.2,3.5);

\draw [color=black](2.5,0)..controls (3,1) and (3.8,1.5)..(6.5,3.5);
\node [below] at(6.2,3){$\mathcal{R}_1$};

\draw [color=black](3.5,0)..controls (3.4,1) and (3.2,1.5)..(4.8,3.5);

\draw [color=black](2.5,0)..controls (2.6,1) and (2.1,1.5)..(3.2,3.5);
\node [below] at(3.5,3){$\mathcal{R}_{\bar2}$};

\draw [color=black](3.5,0)..controls (3,1) and (1.8,2.6)..(1,3.5);
\node [below] at(1.1,3){$\mathcal{R}_5$};

\draw [color=black] (2.5,0)..controls (1.2,2.2) and (0.5,2.8)..(-0.5,3.5);
\end{tikzpicture}
\caption{\small\bf Separation of three characteristic strips.}
\end{figure}
Fortunately, according to the decomposed system \eqref{eqrho0}-\eqref{eqv0}, the interaction terms of the $i^{\text{th}}$ and the $m^{\text{th}}$ characteristic waves vanish for $i,m\in\{2,3,4\}$. In particular, since $\nabla_\Phi\lambda_i=(1,0,0,0,0)$ and $r_{m1}=0$, one can see that all $c_{im}^i=\nabla_\Phi\lambda_i\cdot r_m=0$ for $i,m\in\{2,3,4\}$. Since there is the factor $\lambda_i-\lambda_m$ in $\gamma_{im}^i$, we further get $\gamma_{im}^i=0$ for $i,m\in\{2,3,4\}$. And for $i\in\{1,\cdots,5\}$ and $k,m\in\{2,3,4\}$, similarly we have $\gamma_{km}^i=0$.

Proceeding in the same fashion as in \cite{an,an2}, we introduce the following quantity:
\begin{equation}
V(t)=\max\Big\{\sup_{(x',t')\notin\mathcal{R}_1,\atop 0\leq t'\leq t}|w_1(x',t')|,\sup_{(x',t')\notin\mathcal{R}_5,\atop 0\leq t'\leq t}|w_5(x',t')|,\ V_{\bar{2}}(t)\Big\},
\end{equation}
where
\begin{align*}
V_{\bar{2}}(t):=\sup_{(x',t')\notin\mathcal{R}_{\bar{2}},\atop 0\leq t'\leq t}\{|w_2(x',t')|,|w_3(x',t')|,|w_4(x',t')|\}.
\end{align*}
With the same initial data designed in \eqref{dataw0} and \eqref{data2}, for $t\in[0,T_\eta^*)$, we get the following estimates in a similar manner as in Section \ref{aprior}
 \begin{align}\label{bound0}
  S(t)=O(1),\ J(t)=O(W_0^{(\eta)}),\
 V(t)=O\Big(\eta [W_0^{(\eta)}]^2\Big),\ \bar{U}(t)&=O(\eta W_0^{(\eta)}),
\end{align}
and
\begin{equation} \label{Tshock0}
 \frac1{(1+\varepsilon)^4c_{11}^1(0)W_0^{(\eta)}}\leq T_\eta^*\leq\frac{1}{(1-\varepsilon)^4c_{11}^1(0)W_0^{(\eta)}}.
\end{equation}
Here $T_\eta^*$ is the shock formation time. In particular, a shock forms as $t\rightarrow T_\eta^*$, i.e.,
\begin{equation*}
  \lim_{t\rightarrow T_\eta^*} \rho_1(z_0,t)=0.
\end{equation*}
Via the same argument as in Section \ref{ill}, we have that the $H^1$ norm of the solution to \eqref{0MHD} blows up at $T_\eta^*$. Hence the $H^\frac74$ ill-posedness stated in Theorem \ref{2D} is also true when $\kappa=0$.

\section{$H^\frac74$ ill-posedness for 2D compressible Euler equations} \label{euler ill}
When the  magnetic field vanishes, i.e., $H_1=H_2=0$, the MHD system \eqref{MHD} reduces to the compressible Euler system:
\begin{equation}\label{euler}
\left\{\begin{split}
&\partial_t\varrho+\nabla\cdot(\varrho u)=0,\\
&\varrho\{\partial_t+(u\cdot\nabla)\}u+\nabla p=0,\\
&\partial_t S+(u\cdot\nabla)S=0.
\end{split}
\right.
\end{equation}
Now define
$$\tilde{\mathcal{U}}(x_1,x_2,t)=(u_1,u_2,\varrho,S)^T(x_1,x_2,t).$$
And we denote its planar symmetric solution to be
$$\tilde{\Phi}:=(u_1,u_2,\varrho-1,S)^T(x_1,t)=(u_1,u_2,\varrho-1,S)^T(x,t).$$
Under plane symmetry, the 2D compressible Euler system \eqref{euler} then reads
\begin{equation} \label{ppeuler}
\partial_t\tilde{\Phi}+A(\tilde{\Phi})\partial_x\tilde{\Phi}=0
\end{equation}
with
\begin{equation} \label{matreuler}
A(\tilde{\Phi})=\begin{pmatrix}
u_1 & 0 & c^2/\varrho & c^2/\gamma \\
0 & u_1 & 0 & 0 \\
\varrho & 0 & u_1 &  0 \\
0 &  0& 0 &  u_1
\end{pmatrix}.
\end{equation}
By a direct calculation, we have the eigenvalues of \eqref{matreuler} are
\begin{equation} \label{egveuler}
\lambda_1=u_1+c,\quad\lambda_2=\lambda_3=u_1,\quad \lambda_4=u_1-c.
\end{equation}
This indicates that the Euler system under planar symmetry is not strictly hyperbolic. We then choose the right eigenvectors as:
\begin{equation} \label{rveceuler}
\begin{split}
&r_1=\left(\begin{array}{cc}1\\0\\\varrho/c\\0\end{array}\right),
r_2=\left(\begin{array}{cc}0\\1\\0\\0\end{array}\right),
r_3=\left(\begin{array}{cc}0\\0\\1\\-\gamma/\varrho\end{array}\right),
r_5=\left(\begin{array}{cc}1\\0\\-\varrho/c\\0\end{array}\right).
\end{split}
\end{equation}
The left eigenvectors are set dual to the right ones, i.e.,
\begin{equation}
l_1=(\frac12,0,\frac{c}{2\varrho },\frac{c}{2\gamma}),\quad l_2=(0,1,0,0),\quad
l_3=(0,0,0,-\frac{\varrho}{\gamma}),\quad\quad l_4=(\frac12,0,-\frac{c}{2\varrho },-\frac{c}{2\gamma}).
\end{equation}

Then we derive the decomposed system of $w_i,\rho_i,v_i$ as in \eqref{eqrho0}-\eqref{eqv0} with bounded coefficients $c_{im}^i,\gamma_{im}^i,\gamma_{km}^i$. In particular, the first family of characteristics is genuinely nonlinear because
\begin{equation*}
\begin{split}
c_{11}^1(0)=&\nabla_\Phi\lambda_1(0)\cdot r_1(0)
=\Big(1,0,(\gamma-1)\sqrt{A\gamma},1\Big)\cdot\Big(1,0,\frac{1}{\sqrt{A\gamma}},0\Big)^T
=\gamma>0.
\end{split}
\end{equation*}
Note that the interaction terms of the second and third characteristic waves vanish. The desired $H^\frac74$ ill-posedness follows from the argument as in Section \ref{h10}.
\begin{remark}
In view of the low-regularity local well-posedness result in \cite{zhang} by Zhang and our ill-posedness result, with respect to the regularity of fluid velocity and density, $H^\frac{7}{4}$ naturally serves as the threshold for ill-posedness of 2D compressible Euler equations. Moreover, we introduce the vorticity $\omega=-\partial_{x_2}u_1+\partial_x u_2$. Under plane symmetry, the vorticity $\omega=\partial_x u_2$ is regular since
\begin{equation*}
    \partial_x u_2=\sum_{k=1}^4 w_k r_{k2}=w_2.
\end{equation*}
The system \eqref{ppeuler} therefore allows non-trivial vorticity and entropy.
\end{remark}

\appendix
\section{Finite speed of propagation for 2D ideal MHD} \label{FSP}
In this appendix, we show the finite speed of propagation for the 2D ideal compressible MHD system without any symmetry condition. There are two main steps:
\begin{itemize}
\item Step 1. rewrite the system as a quasilinear symmetric hyperbolic system.\\
 \item Step 2. prove local energy estimates for the new system.
 \end{itemize}

 {\bf Step 1.} In this paper, we consider the 2D ideal compressible MHD system \eqref{MHD} with variables $(u_1,u_2,\varrho,H_1,H_2,S)$:
\begin{equation*}
\left\{\begin{split}
&\partial_t\varrho+\nabla\cdot(\varrho u)=0,\\
&\varrho\{\partial_t+(u\cdot\nabla)\}u+\nabla p+\frac{\mu_0}{2} \nabla |H|^2-\mu_0 (H \cdot \nabla) H=0,\\
&\partial_t H+(u\cdot \nabla)H-(H\cdot \nabla) u +H (\nabla\cdot u)=0,\\
&\partial_t S+(u\cdot\nabla)S=0,\\
&\nabla\cdot H=0,\\
&p=A\varrho^\gamma e^S.
\end{split}
\right.
\end{equation*}

We can transform the system \eqref{MHD} into a symmetric hyperbolic system by introducing new variables $U=(u_1,u_2,p-A,H_1,H_2,S)$ for the following system:
\begin{equation}\label{pMHD fsp}
\left\{\begin{split}
&\varrho\{\partial_t+(u\cdot\nabla)\}u+\nabla p+\frac{\mu_0}{2} \nabla |H|^2-\mu_0 (H \cdot \nabla) H=0,\\
&\frac{1}{\varrho c^2}\partial_t p+\nabla\cdot u+\frac{1}{\varrho c^2}(u\cdot\nabla)p=0,\\
&\mu_0\partial_t H+\mu_0(u\cdot \nabla)H-\mu_0(H\cdot \nabla) u +\mu_0H (\nabla\cdot u)=0,\\
&\partial_t S+(u\cdot\nabla)S=0.
\end{split}
\right.
\end{equation}
With these new variables, the above system \eqref{pMHD fsp} becomes
\begin{equation}\label{hpform}
A_0\partial_tU+A_1\partial_{x_1}U+A_2\partial_{x_2}U=0,
\end{equation}
where
\begin{equation*}
\begin{split}
&A_0=\text{diag}\{\varrho,\varrho,\frac{1}{\varrho c^2},\mu_0,\mu_0,1\},\\
&A_1=\begin{pmatrix}
\varrho u_1 & 0 &1 &0 & \mu_0H_2 &  0 \\
0 & \varrho u_1 & 0 & 0&-\mu_0H_1 & 0 \\
1 & 0 & \frac{u_1}{\varrho c^2} & 0  & 0 &0\\
0&0&0&\mu_0 u_1&0&0\\
\mu_0 H_2 & -\mu_0H_1 & 0 &0&\mu_0u_1 & 0\\
0 & 0  & 0 &0& 0 &  u_1
\end{pmatrix},\\
&A_2=\begin{pmatrix}
\varrho u_2 & 0   &0 & -\mu_0H_2 &  0&0 \\
0 & \varrho u_2 &1&\mu_0H_1 & 0&0 \\
0 & 1 & \frac{u_2}{\varrho c^2} & 0  & 0 &0\\
-\mu_0 H_2 & \mu_0H_1 &  0&\mu_0u_2&0 & 0\\
0&0&0&0&\mu_0 u_2 &0\\
0 & 0  & 0 &0& 0 &  u_2
\end{pmatrix}.\\
\end{split}
\end{equation*}
When there is no vacuum region (i.e., $\varrho>0$), $A_0$ is a positive diagonal matrix, $A_1$ and $A_2$ are symmetric matrices.

{\bf Step 2.}
Define $Q(\lambda,\xi)=\lambda I-\xi_1A_0^{-1}A_1-\xi_2A_0^{-1}A_2$, where $\xi=(\xi_1,\xi_2)\in S^1=\{\xi\in\mathbb{R}^2:|\xi|=1\}$ and $\lambda\in\mathbb{R}$. For all $\xi$, the characteristic equation
\begin{equation}
\det Q(\lambda, \xi)=0
\end{equation}
has six real roots $\lambda_1(\xi),\cdots,\lambda_6(\xi)$. In detail, we have
\begin{equation}
\begin{split}
&\lambda_1=\xi\cdot u+\frac{\mu_0 H^2+\varrho c^2+\sqrt{(\mu_0 H^2+\varrho c^2)^2-4\mu_0\varrho c^2(\xi\cdot H)^2}}{2\varrho},\\
&\lambda_2=\xi\cdot u+\frac{\mu_0 H^2+\varrho c^2-\sqrt{(\mu_0 H^2+\varrho c^2)^2-4\mu_0\varrho c^2(\xi\cdot H)^2}}{2\varrho},\\
&\lambda_3=\lambda_4=\xi\cdot u,\\
&\lambda_5=\xi\cdot u-\frac{\mu_0 H^2+\varrho c^2-\sqrt{(\mu_0 H^2+\varrho c^2)^2-4\mu_0\varrho c^2(\xi\cdot H)^2}}{2\varrho},\\
&\lambda_6=\xi\cdot u-\frac{\mu_0 H^2+\varrho c^2+\sqrt{(\mu_0 H^2+\varrho c^2)^2-4\mu_0\varrho c^2(\xi\cdot H)^2}}{2\varrho},\\
\end{split}\end{equation}
where $H^2=H_1^2+H_2^2$. Note that $\lambda_1$ and $\lambda_6$ are the largest and smallest of the six roots, respectively. For quasilinear symmetric hyperbolic system, according to \cite{sideris}, the finite propagating-speed property holds. Here, in particular for the 2D ideal MHD system, we summarize the proof in the following proposition.
\begin{prop} 
Let $U(x_1,x_2,t)$ be a $C^1$ function on $\mathbb{R}^2\times[0,T]$. 
Define
$$\Omega(\xi,s;T)=\{(x_1,x_2,t): x\cdot \xi\geq s+\lambda_1(\xi)t,\  0\leq t\leq T\}.$$
Suppose $U(x_1,x_2,t)$ solves system \eqref{pMHD fsp} in $\Omega$ with initial data $U_0(x_1,x_2)=0$. Then, we have $U(x_1,x_2,t)\equiv0$ throughout $\Omega(\xi,s;T)$.
\end{prop}
{\it Proof.} Without loss of generality, we assume $s=0$. Let $(x_0,t_0)\in\Omega(\xi,0;T)$. Take $\xi=e_1,e_2$, where $e_1=(1,0)$ and $e_2=(0,1)$.
Define
$$E(t)=\{x:\lambda_6(e_i)(t_0-t)\leq(x_0-x)\cdot e_i\leq\lambda_1(e_i)(t_0-t);\ i=1,2\}$$
with $0\leq t\leq t_0$. Thus, we have that
$$E=\{(x,t): x\in E(t), 0\leq t\leq t_0\}$$
is a pyramid with vertex $(x_0,t_0)$, which is contained entirely within $\Omega(\xi,0;T)$.
The outer normal to $E$ along the lateral surface $$E_i^+=\{(x,t):(x_0-x)\cdot e_i=\lambda_1(e_i)(t_0-t)\}$$
is $$n_i^+=(-e_i,\lambda_1(e_i)).$$
Respectively, along the lateral surface $$E_i^-=\{(x,t):(x_0-x)\cdot e_i=\lambda_6(e_i)(t_0-t)\}$$
the outer normal to $E$ is $$n_i^-=(e_i,-\lambda_6(e_i)).$$
We consider the linear partial differential operator around $U=0$, which corresponds to a non-vacuum state, i.e., $\varrho=1$,
$$P=\bar{A_0}\partial_t+\bar{A_1}\partial_{x_1}+\bar{A_2}\partial_{x_2},$$
where $\bar{A_i}=A_i(0)$ $(i=0,1,2)$. Since $\bar{A_i}$ are symmetric, there holds the following energy identity
$$(PU,U)=\frac12\Big\{\partial_t(\bar{A_0}U,U)+\partial_{x_1}(\bar{A_1}U,U)+\partial_{x_2}(\bar{A_2}U,U)\Big\}.$$
Since $U=0$ on $E(0)$, by the divergence theorem, for $0<t<t_0$, we have
\begin{equation}\label{energy}
\begin{split}
2\int_0^t\int_{E(\tau)}(PU,U)dxd\tau=&\int_{E(t)}(\bar{A_0}U,U)_{\tau=t}dx+\sum_{i=1,2}\int_0^t\int_{E_i^+}(\bar{A_0}Q(\lambda_1,e_i)U,U)d\sigma_i^+d\tau\\
&-\sum_{i=1,2}\int_0^t\int_{E_i^-}(\bar{A_0}Q(\lambda_6,e_i)U,U)d\sigma_i^-d\tau
\end{split}
\end{equation}
with $d\sigma_i^{\pm}$ denoting the surface element on $E_i^{\pm}$.

Note that $\bar{A_0}$ is positive, there hold $\bar{A_0}Q(\lambda_1,e_i)\geq 0$ and  $\bar{A_0}Q(\lambda_6,e_i)\leq 0$. So from \eqref{energy}, we conclude that
\begin{equation}\label{energycontrol}
\int_{E(t)}|U(t)|^2dx\leq 2a\int_0^t\int_{E(\tau)}(PU,U)dxd\tau,
\end{equation}
where $\frac1a$ is the lower bound for $\bar{A_0}$.

Employing $A_0\partial_tU+A_1\partial_{x_1}U+A_2\partial_{x_2}U=0$, in $\Omega(\xi,0;T)$ (which contains $E$), we have that,
\begin{equation}
PU=(\bar{A_0}-A_0)\partial_tU+(\bar{A_1}-A_1)\partial_{x_1}U+(\bar{A_2}-A_2)\partial_{x_2}U.
\end{equation}
Because $U$ is $C^1$ in $E$, by mean value theorem, it holds in $E$ that
\begin{equation}
|(PU,U)|\leq C_E|U|^2,
\end{equation}
where $C_E$ is a constant depending on the magnitude of $U$ and its derivatives inside $E$.
Hence, from \eqref{energycontrol} it follows that
\begin{equation}
\int_{E(t)}|U(t)|^2dx\leq C_E\cdot a\int_0^t\int_{E(\tau)}|U(\tau)|^2dxd\tau,\qquad\text{for all}\  0\leq t\leq t_0.
\end{equation}
Therefore, by Gr\"{o}nwall's inequality, we have $U\equiv0$ on all of $E$.

\subsection*{Conflict of interest statement}
The authors declare that they have no conflict of interest.

\subsection*{Data availability statement}
Data sharing is not applicable to this article as no datasets were generated or analyzed during the current study.

\end{document}